%% file: N-gon.v5b.tex
\documentclass[11pt,leqno]{amsart}

\usepackage{latexsym}
\usepackage{amssymb}
\usepackage{amsmath}
\usepackage{amsthm}
\usepackage{verbatim}
\usepackage{pdfsync}
\usepackage[margin=1.2in]{geometry}

\usepackage{color}
\usepackage{tikz}
\usepackage[all]{xy}

\numberwithin{equation}{section}  

\begin{document}



\newcommand{\zxz}[4]{\begin{pmatrix} #1 & #2 \\ #3 & #4 \end{pmatrix}}
\newcommand{\abcd}{\zxz{a}{b}{c}{d}}
\newcommand{\kzxz}[4]{\left(\begin{smallmatrix} #1 & #2 \\ #3 &
#4\end{smallmatrix}\right) }
\newcommand{\kabcd}{\kzxz{a}{b}{c}{d}}

\input KRY.macros

\newcommand{\EE}{{\mathbb E}}

\newcommand{\OO}{\text{\rm O}}
\newcommand{\UU}{\text{\rm U}}

\newcommand{\OK}{O_{\smallkay}}
\newcommand{\DI}{\mathcal D^{-1}}

\newcommand{\pre}{\text{\rm pre}}

\newcommand{\Bor}{\text{\rm Bor}}
\newcommand{\Rel}{\text{\rm Rel}}
\newcommand{\rel}{\text{\rm rel}}
\newcommand{\Res}{\text{\rm Res}}
\newcommand{\TG}{\widetilde{G}}

\newcommand{\OL}{O_{\Lambda}}
\newcommand{\OLB}{O_{\Lambda,B}}

\newcommand{\p}{\varpi}

\newcommand{\cutter}{\vskip .1in\hrule\vskip .1in}

\parindent=0pt
\parskip=10pt
\baselineskip=14pt

\newcommand{\PP}{\mathcal P}
\renewcommand{\OO}{\mathcal O}
\newcommand{\BB}{\mathbb B}
\newcommand{\OBB}{O_{\BB}}
\newcommand{\Max}{\text{\rm Max}}
\newcommand{\Opt}{\text{\rm Opt}}
\newcommand{\OH}{O_H}

\newcommand{\phhat}{\widehat{\phi}}
\newcommand{\thetahat}{\widehat{\theta}}

\newcommand{\lbold}{\text{\boldmath$\l$\unboldmath}}
\newcommand{\abold}{\text{\boldmath$a$\unboldmath}}
\newcommand{\cbold}{\text{\boldmath$c$\unboldmath}}
\newcommand{\aabold}{\text{\boldmath$\a$\unboldmath}}
\newcommand{\gbold}{\text{\boldmath$g$\unboldmath}}
\newcommand{\obold}{\text{\boldmath$\o$\unboldmath}}
\newcommand{\fbold}{\text{\boldmath$f$\unboldmath}}
\newcommand{\rbold}{\text{\boldmath$r$\unboldmath}}
\newcommand{\ffbold}{\und{\fbold}}

\newcommand{\boldgamma}{\text{\boldmath$\gamma$\unboldmath}}

\newcommand{\deltaBB}{\delta_{\BB}}
\newcommand{\kappaBB}{\kappa_{\BB}}
\newcommand{\aboldBB}{\abold_{\BB}}
\newcommand{\lboldBB}{\lbold_{\BB}}
\newcommand{\gboldBB}{\gbold_{\BB}}
\newcommand{\bbold}{\text{\boldmath$\b$\unboldmath}}

\newcommand{\fff}{\phi}

\newcommand{\spp}{\text{\rm sp}}

\newcommand{\pob}{\mathfrak p_{\bold o}}
\newcommand{\kob}{\mathfrak k_{\bold o}}
\newcommand{\gob}{\mathfrak g_{\bold o}}
\newcommand{\pobp}{\mathfrak p_{\bold o +}}
\newcommand{\pobm}{\mathfrak p_{\bold o -}}


\newcommand{\bb}{\frak b}

\newcommand{\bbbold}{\text{\boldmath$b$\unboldmath}}

\renewcommand{\ll}{\,\frak l}
\newcommand{\uC}{\underline{\Cal C}}
\newcommand{\uZZ}{\underline{\ZZ}}
\newcommand{\B}{\mathbb B}
\newcommand{\CL}{\text{\rm Cl}}

\newcommand{\pp}{\frak p}

\newcommand{\OKp}{O_{\smallkay,p}}

\renewcommand{\top}{\text{\rm top}}

\newcommand{\bF}{\bar{\mathbb F}_p}

\newcommand{\beq}{\begin{equation}}
\newcommand{\eeq}{\end{equation}}


\newcommand{\Dl}{\Delta(\l)}
\newcommand{\mm}{{\bold m}}

\newcommand{\FD}{\text{\rm FD}}
\newcommand{\LDS}{\text{\rm LDS}}

\newcommand{\dcM}{\dot{\Cal M}}
\newcommand{\bpm}{\begin{pmatrix}}
\newcommand{\epm}{\end{pmatrix}}

\newcommand{\GW}{\text{\rm GW}}

\newcommand{\uk}{\bold k}
\newcommand{\uo}{\text{\boldmath$\o$\unboldmath}}

\newcommand{\uz}{\und{\zeta}}
\newcommand{\duz}{\und{\dot\zeta}}
\newcommand{\ub}{\und{b}}
\newcommand{\uB}{\und{B}}
\renewcommand{\uC}{\und{C}}
\newcommand{\xp}{x_+}
\newcommand{\xm}{x_-}
\newcommand{\tu}{\tilde u}
\newcommand{\sspan}{\text{\rm span}}

\newcommand{\kk}{\mathfrak k}
\newcommand{\tio}{\tilde\o}
\newcommand{\Fock}{\text{\rm Fock}}
\newcommand{\AO}{{\bf AO}}
\newcommand{\qq}{{\bold q}}
\newcommand{\Ps}{\Psi}

\newcommand{\dbs}[1]{\frac{\d \uB_{#1}}{\d s}}
\newcommand{\dbt}[1]{\frac{\d \uB_{#1}}{\d t}}

\newcommand{\OFD}{\text{\rm OFD}}
\newcommand{\Erf}{\text{\rm Erf}}
\newcommand{\Erfc}{\text{\rm Erfc}}
\newcommand{\Arctan}{\text{\rm Arctan}}

\newcommand{\CC}{\mathcal C}
\newcommand{\bC}{\bold C}

\newcommand{\simp}{\text{\rm Simp}}

\newcommand{\mrE}{\mathring{E}}

\newcommand{\zzeta}{\text{\boldmath$\zeta$\unboldmath}}

\newcommand{\vbold}{\text{\boldmath$v$\unboldmath}}
\newcommand{\sbold}{\text{\boldmath$\s$\unboldmath}}
\newcommand{\zbold}{\text{\boldmath$\z$\unboldmath}}
\newcommand{\Tbold}{\text{\boldmath$T$\unboldmath}}
\newcommand{\hbold}{\text{\boldmath$h$\unboldmath}}

\newcommand{\qbold}{\text{\boldmath$q$\unboldmath}}

\newcommand{\wbold}{\text{\boldmath$w$\unboldmath}}

\newcommand{\Reg}{\text{Reg}}

\newcommand{\orr}{\text{\rm or}}

\newcommand{\aaa}{\a}

\newcommand{\FF}{\mathcal F}
\newcommand{\Cdod}{\mathcal C^{\text{\rm dod}}}
\newcommand{\Ddod}{\mathcal D^{\text{\rm dod}}}
\newcommand{\RR}{\mathcal R}

\newcommand{\nubold}{\text{\boldmath$\nu$\unboldmath}}

\newcommand{\wnat}{\wbold}

\newcommand{\newmod}{\, \text{\rm mod}\, }

\newcommand{\hh}{\mathbb H}

\newcommand{\cross}{\text{\boldmath$\times$\unboldmath}}

\newcommand{\now}{\count0=\time 
\divide\count0 by 60
\count1=\count0
\multiply\count1 by 60
\count2= \time
\advance\count2 by -\count1
\the\count0:\the\count2}


\title{The case of an $N$-gon}

\author{Jens Funke  and  Stephen Kudla}

\dedicatory{To Don Zagier, with our admiration}

\maketitle

\section{Introduction}

In this note we provide a construction of the indefinite theta series attached to $N$-gons in the symmetric space 
of an indefinite inner product space of signature $(m-2,2)$ following the suggestions of section C in the recent paper of 
Alexandrov, Banerjee, Manschot, and Pioline, \cite{ABMP-II}. 
We prove the termwise absolute convergence of the holomorphic mock modular part
of these series and also obtain an interpretation of the coefficients of this part as linking numbers. Thus we prove the convergence conjecture of 
\cite{ABMP-II} provided none of the vectors in the collection $\CC=\{C_1,\dots, C_N\}$ is a null vector. It should  be noted that the use of linking numbers 
and a homotopy argument 
eliminates the need for an explicit parametrization of a surface $S$ spanning the $N$-gon that was used in an essential way in our previous work \cite{FK-I}.
We indicate 
how our method could be carried over to a more general situation for signature $(m-q,q)$ where higher homotopy groups are now involved. 
In the last section, we apply the method to the case of a dodecahedral cell in the symmetric space of a quadratic form of signature $(m-3,3)$.

\subsection{Indefinite theta series}
The construction of theta series for lattices with indefinite quadratic forms continues to be a topic of considerable interest. 
If $(L,Q)$ is a lattice of rank $m$ with integer valued positive definite quadratic form, the classical theta series 
\beq\label{naive}
\theta_\mu(\tau) = \sum_{x\in \mu+L} \qq^{Q(x)}, \qquad \qq = e(\tau) = e^{2\pi i \tau},
\eeq
is termwise absolutely convergent and defines a holomorphic modular form of weight $\frac12 m$, by Poisson summation. Here $\tau=u+iv\in \H$, the upper half plane, and $\mu\in L^\vee$, the dual lattice of $L$.
For a quadratic lattice  $(L,Q)$ of signature $(p,q)$ with $pq>0$, the series (\ref{naive}) no longer converges in general due to (i) the presence of lattice vectors $x$ 
with $Q(x)<0$ and (ii) the existence infinitely many lattice vectors $x$ with a given value $Q(x)=n>0$, due, for example, 
for the existence of an infinite group $\Gamma_L$ of isometries of $L$. 
To handle this case, the classic construction due to Siegel introduces an auxiliary variable 
$$z\in D(V) := \{\ z\in  \text{Gr}^o_q(V)\,\mid\, Q\vert_z<0\ \},$$
where $\text{Gr}^o_q(V)$ is the Grassmannian of oriented $q$-planes in $V = L\tt_\Z\R$. 
The variable $z$ determines a positive definite inner product 
$$(x,x)_z = (x,x) + 2 R(x,z), \qquad R(x,z) = -(\pr_z(x),\pr_z(x)),$$
which coincides with $(\ ,\ )$ on $z^\perp$ and with $-(\ ,\ )$ on $z$. Here $\pr_z$ denotes the orthogonal projection to $z$.
Siegel's theta series 
$$\theta_\mu(\tau,z) = \sum_{x\in \mu+L}  e^{-2\pi R(x,z) v}\,\qq^{Q(x)}$$
is a non-holomorphic modular form of weight $\frac12(p-q)$ in $\tau$ and a smooth function of $z\in D$.  
There are also `geometric' theta series introduced in \cite{KM.I} and \cite{KM.II}, which are valued in $A^q(D)$, the space of smooth $q$-forms on $D$, and  
are of weight $\frac12(p+q)$ in $\tau$. 
Both these series and Siegel's are invariant under the action of an arithmetic subgroup $\Gamma_L\, \subset \, O(V)$ on $D$. As a consequence one can use the Siegel-type theta series and the geometric theta series as integral kernel to lift (automorphic) forms for one group to the other. See section 2 for a more detailed review of these theta series.

Another approach to the construction of indefinite theta series\footnote{This was suggested by Zagier in conversations with the second author at Maryland in the early 1980's.} involves restricting the summation to lattice vectors lying in suitable cones on which the quadratic form takes positive values.  When they are termwise absolutely convergent,  such sums define holomorphic functions of $\tau$. The standard holomorphic Eisenstein series $E_k$ series can be viewed as such an indefinite theta series of signature $(1,1)$ but in general they are rarely modular due to the lack of a suitable Poisson summation formula. Another non-modular example for signature $(1,1)$ was constructed by Zagier in section 2.3 in \cite{HZ} in his celebrated joint work with Hirzebruch where the holomorphic generating series was completed by a non-holomorphic part to obtain a modular object. 

A more systematic understanding of the nature of such series emerged from the seminal work of Zwegers \cite{zwegers}; for an overview see \cite{Zag10}. In the case of a quadratic lattice of signature $(m-1,1)$, 
suppose that $C$ and $C'$ are negative vectors in $V =  L\tt_\Z\R$ lying in the same nap of the light cone, i.e.,
such that $(C,C')<0$.  Then Zwegers considered  the series
\beq\label{Zweg-signs}
\vartheta_\mu(\tau,\mathcal C) = \sum_{\substack{x\in \mu+L} } \big(\  \sgn((x,C))-\sgn((x,C'))\ \big)\, \qq^{Q(x)},
\eeq
and showed that it is termwise absolutely convergent and defines a mock modular form, i.e., 
it has a non-holomorphic modular completion of weight $\frac12m$
\begin{align*}
\hat{\vartheta}_\mu(\tau, \mathcal C)= \vartheta_\mu(\tau,\mathcal C)
-\sum_{x\in \mu+L} \big(\,E^c((x,\uC)\sqrt{2v}) - E^c((x,\uC')\sqrt{2v})\,\big)\, \qq^{Q(x)},
\end{align*}
involving the complementary error function
$E^c(u) = 2\,\sgn(u)\int_{|u|}^\infty e^{-\pi t^2}\,dt$.
%

Alexandrov, Banerjee, Manschot, and Pioline, \cite{ABMP}, established the analog of Zwegers' result for signature $(m-2,2)$. For a collection $\mathcal{C}=\{C_1,C_2,C_3,C_4\}$ of four negative vectors in $V$ satisfying certain conditions they showed that
\[
\hat{\vartheta}_\mu(\tau, \mathcal C)=  \sum_{x\in \mu+L} \sum_{j=1}^4 E_2(C_j,C_{j+1},x\sqrt{2})\,\qq^{Q(x)}
\]
 is the non-holomorphic modular completion of the holomorphic generating series
\beq\label{ABMP-signs}
 \vartheta_\mu(\tau,\mathcal C) = \sum_{\substack{x\in \mu+L}} \sum_{j=1}^4\sgn(x,C_j)\,\sgn(x,C_{j+1})\, \qq^{Q(x)}.
\eeq
Here 
\[
E_2(x;c,c') = \int_{z} e^{\pi (y-\pr_z(x),y-\pr_z(x))}\,\sgn(y,c)\,\sgn(y,c')\,dy,
\]
is the higher error function defined in \cite{ABMP} with $z$ the negative $2$-plane spanned by $c$ and $c'$ with projection $\pr_z(x)$ onto $z$ and an appropriately normalized measure $dy$. In a subsequent paper \cite{ABMP-II}, the authors associate similar generating series to a collection $\mathcal{C}=\{C_1,\dots,C_N\}$ of $N$ negative vectors and conjecture their modularity. Nazaroglu \cite{Naz16} discussed the generalization to arbitrary signature $(m-q,q)$, while Raum \cite{Rau16} considered more general positive polyhedral cones. All these results employ a modularity criterion for indefinite theta series obtained by Vign\'eras \cite{Vig77}. 

It is important to note that the sign functions occurring in the summations (\ref{Zweg-signs}) and (\ref{ABMP-signs}) actually imply that only terms with $Q(x)>0$ contribute. While this is easy to 
see in (\ref{ABMP-signs}), it is not at all evident
in  general.

In \cite{kudla-int,FK-I,FK-II} we developed a different approach based on the geometric theta series of Kudla-Millson \cite{KM.I,KM.II}. Recall that the Kudla-Millson theta series 
\[
\vartheta_\mu(\tau,\ph_{KM}) = \sum_{x\in L+\mu}\ph_{KM}(x,\tau)
\]
is a closed differential $q$-form on the symmetric space $D$ associated to $V$ which in $\tau$ transforms like a modular form of weight $m/2$. Then for certain collections $\mathcal{C}$ of negative vectors we associate a (compact) $q$-cell $S(\mathcal{C})$ in $D$ and consider the integral
\[
 \int_{S(\mathcal{C})} \vartheta_\mu(\tau,\ph_{KM}),
\]
which inherits the modularity from $\theta_\mu(\tau,\ph_{KM})$. On the other hand, under certain conditions, we can compute the theta integral explicitly as a combination of 
generalized error functions and give a geometric interpretation of the holomorphic generating series. In particular, we obtained new proofs for the results obtained in \cite{zwegers}, \cite{ABMP}, and \cite{Naz16} for `cubical' input data, but we also established a new result when $\CC$ gives rise to geodesic $q$-simplices.

\subsection{The case of the $N$-gon}  In the present paper, we consider another type of sign function 
introduced in \cite{ABMP-II}. 
Suppose that $\sig(V)= (m-2,2)$ and that $L$ is a lattice in $V$ such that 
with $L\subset L^\vee$. 
Let $\CC= \{C_1,\dots,C_N\}$ be a collection of $N$ negative vectors in $V$ such that the oriented spans $z_j=[C_j,C_{j+1}]$ are negative $2$-planes
all lying in the same connected component of $D$, say $D^+$. This amounts to the set of conditions
\begin{align}
(C_j,C_j)&<0,\label{N-gon-c1}\\
\nass
(C_j,C_j)(C_{j+1},C_{j+1}) - (C_j,C_{j+1})^2&>0,\label{N-gon-c2}\\
\nass
(C_j,C_j)(C_{j-1},C_{j+1}) - (C_j,C_{j-1})(C_j,C_{j+1}) &<0,\label{N-gon-c3}
\end{align}
on inner products that we will sometimes refer to as the {\bf $N$-gon conditions}. Here the first two conditions are equivalent to $z_j$ being a negative $2$-plane, while the third ensure that all $z_j$ are indeed in the same component of $D$. Note that the last of these conditions is equivalent to $(-C_{j-1\perp j},C_{j+1\perp j})<0$, where for example, adopting the notation of \cite{ABMP}, 
$$C_{j-1\perp j} = C_{j-1} - \frac{(C_{j-1},C_j)}{(C_j,C_j)}\,C_j$$
 is the projection of $C_{j-1}$ to $C_j^{\perp}$. In particular,  the projections $-C_{j-1\perp j}$ and $C_{j+1\perp j}$  lie in the 
same nap of the light cone in the space $V_j:=C_j^\perp$ of signature $(m-2,1)$.  Thus 
\beq\label{gamma-j}
\gamma_j(s) := [C_j, (s-1)C_{j-1}+s C_{j+1}],\qquad s\in [0,1],
\eeq
defines a geodesic curve in $D'_{C_j} \cap D^+$ joining $z_{j-1}=[C_{j-1},C_j]$ and $z_j=[C_j,C_{j+1}]$, 
and  
\beq\label{gamma-C}
\gamma(\CC)= \sum_j \gamma_j
\eeq
is a closed and piecewise smooth (not necessarily simple) curve in $D^+$ with the $z_j$'s as it `vertices'. This is the N-gon of the title. With this let $S=S(\mathcal{C})$ be any oriented $2$-cell with boundary $\gamma(\CC)$. 

For $x\in V$, define the quantities
\begin{align}
\varepsilon(x;\mathcal C) &:= \sum_{j=1}^N \sgn(x,C_j)\,\sgn(x,C_{j+1}) - \sgn(\vbold,C_j)\,\sgn(\vbold,C_{j+1}),\\
\noalign{\vskip -10pt and }
\wnat(\CC) &:= -\sum_{j=1}^N  \sgn(\vbold,C_j)\,\sgn(\vbold,C_{j+1}), \label{def-wC-nat}
\end{align}
where $\vbold$ is any negative vector in $V$.  Here $\sgn(0)=0$. We will see that $\wnat(\CC)$ is independent of the choice of $\vbold$. 
Note that $|\varepsilon(x;\mathcal C)|\le 2 N$. 

Our main result concerning the $N$-gon is the following.
\begin{theo} For $\mu\in L^\vee/L$ the series
$$\vartheta_\mu(\tau,\CC) = \sum_{x\in \mu+L} \varepsilon(x;\CC)\,\qbold^{Q(x)}$$
is termwise absolutely convergent. Its modular completion given by 
$$\hat\vartheta_\mu(\tau,\CC) = \sum_{x\in \mu+L} \bigg(\  \wnat(\CC)+ \sum_{j=1}^N E_2(C_j,C_{j+1},x\sqrt{2})\ \bigg)\, \qbold^{Q(x)}$$
has weight $\frac{m}2$. Further,
$$\hat\vartheta_\mu(\tau,\CC)  = 4 \int_{S}\theta(\tau,\ph_{KM}).$$
\end{theo} 

For $N=4$ we will see below one always has $\wnat(\CC)=0$, and one recovers the main result of \cite{ABMP} which was also considered in \cite{kudla-int}. For triangles ($N=3$) one has $\wnat(\CC)=1$, and this is a very special case of the results in \cite{FK-II}. In general, we show $-N+4 \leq \wnat(\CC) \leq N-2$ and $\wnat(\CC) \equiv -N \pmod{4}$. 
The invariant $\wnat(\CC)$ is somewhat curious, and we discuss some of its properties in section~\ref{section-more-more}. 
 
For general $N$, this establishes the convergence conjecture made in \cite{ABMP-II} while also clarifying the role of the constant $\wnat(\CC)$. Finally, our proof will show that, as before, in the generating series $\vartheta_\mu(\tau,\CC) $ of the signs $\varepsilon(x;\CC)$ only terms with $Q(x)>0$ contribute. 

As in our prior work, this result is obtained by computing the integral of the geometric theta kernel
over the oriented $2$-cell $S$ with boundary $\gamma(\CC)$. 
Since $S$ is compact, this integral can be computed termwise, so the key is to compute the local theta integral
$$I^0(x,\CC) = \int_S \ph_{KM}^0(x)$$
for $x\in V$ and the Schwartz $2$-form used to construct the theta kernel. Further note that since $\ph_{KM}$ is closed the integral does not depend on the choice of $S$. 

The result given in Theorem~\ref{theoA} involves some new ideas. In our previous work we needed to impose some additional conditions on $\CC$ which ensured that we could explicitly parameterized the cell $S$ In this paper, we introduce a very general homotopy argument which allows to compute the local theta integral. Furthermore, the argument also shows that Fourier coefficients $\varepsilon(x;\CC)$ of the holomorphic generating series $\vartheta_\mu(\tau,\CC)$ can be interpreted as the linking number between the curve $\gamma(\CC)$ and a divisor $D_x$ in $D$. 



The conjecture made in \cite{ABMP-II} allows for the $C$'s to be rationally null. While we do not consider this case in this paper, we give an example, Theorem~\ref{th:Zagier}, where we explain that previous work of the first author \cite{Funke-diss} realizes Zagier's Eisenstein series \cite{Zagier,HZ} as indefinite theta series associated to signature $(1,2)$! For an outline approach to the general null vector case, see Remark~\ref{null-vector-remark}. 

In section 7 we outline how the results of this paper can be extended to general signature $(m-q,q)$, and in section $8$ we carry this out explicitly for $q=3$ for a geodesic dodecahedron. The cases of a cube and a tetrahedron were discussed in \cite{FK-II}, while other solids can be covered by our methods as well.


\section{Some machinery}

In this section, we review some of the basic machinery involved in the construction of indefinite theta series. 

\subsection{Non-holomorphic theta series} 
Let $V$, $(\ ,\ )$ be an inner product space over $\R$ of signature $(p,q)$ with $pq>0$, and let
$$D=D(V) := \{\ z\in  \text{Gr}^o_q(V)\,\mid\, Q\vert_z<0\ \}$$
be the space of oriented negative $q$-planes in $V$. Note that $D$ has two components $D^{\pm}$. Here $Q(x) = \frac12(x,x)$.  For $z\in D$ and $x\in V$, the associated majorant is  
$$(x,x)_z = (x,x) + 2 R(x,z), \qquad R(x,z) = -(\pr_z(x),\pr_z(x)).$$
The Siegel Gaussian is
$$\ph_0(x)=\ph_0(x,z) = e^{-\pi (x,x)_z} = e^{-\pi (x,x)} \,\ph_0^0(x), \qquad \ph_0^0(x) = e^{-2\pi R(x,z)}.$$
There is also a Schwartz form, see \cite{KM.I},
$$\ph_{KM}(x) = e^{-\pi (x,x)} \,\ph_{KM}^0(x)$$
valued in $A^q(D)$, the space of smooth $q$-forms on $D$.  
These forms are equivariant with respect to the natural action of the 
orthogonal group $G=\text{\rm O}(V)$:
$$\ph_0\in [\,S(V)\tt A^0(D)\,]^G, \qquad \ph_{KM} \in [ \, S(V)\tt A^q(D)\,]^G.$$
Moreover $\ph_{KM}(x)$ is closed, $d\ph_{KM}(x)=0$. 
With respect to the Weil representation action of the metaplectic group $\widetilde{\SL_2(\R)}$ on the Schwartz space $S(V)$, 
the functions $\ph_0$ and $\ph_{KM}$ are eigenfunctions of weights $\frac12(p-q)$ and $\frac12(p+q)$ respectively 
for the subgroup $\widetilde{\SO(2)}$. 

If $V= L\tt\R$ for a quadratic lattice $L$ with dual lattice $L^\vee \supset L$, then the (non-holomorphic) theta series
for $\mu\in L^\vee$ and $\tau=u+iv\in \H$, 
$$\theta_\mu(\tau,\ph_0) = \sum_{x\in L+\mu} \qq^{Q(x)}\, \ph_0^0(v^{\frac12}x),\qquad \qq=e(\tau),$$
and 
$$\theta_\mu(\tau,\ph_{KM}) = \sum_{x\in L+\mu} \qq^{Q(x)}\, \ph_{KM}^0(v^{\frac12}x),$$
have modular transformations of weights $\frac12(p-q)$ and $\frac12(p+q)$ respectively. The second of these 
defines a closed $q$-form on $D$, invariant with respect to the discriminant group  
$$\Gamma_L = \{ \gamma\in \text{\rm O}(V) \mid \gamma L= L,\ \gamma\vert_{L^\vee/L} = \text{id}\ \}.$$ 
More details can be found in section 2 of \cite{FK-II}. 

\subsection{Subspaces of $D$} There are sub-symmetric spaces of $D$ attached non-null vectors in $V$ as follows. 

 If $x\in V$ is a nonzero vector, let
\beq\label{def-Dx}
D_x = \{\, z\in D\,\mid x\in z^\perp\ \} = \{ \, z\in D\mid R(x,z)=0\ \}.
\eeq
This space is only non-empty if $(x,x)>0$, and in that case can be identified with the 
the space $D(V_x)$ of oriented negative $q$-planes in $V_x = x^\perp$. Note that $\sig(V_x)= (p-1,q)$
so that $D_x$ has codimension $q$ in $D$. 

Similarly, for a non-zero vector $x$, let
\beq\label{def-Dx'}
D'_x = \{ \, z\in D\,\mid x\in z\,\} = \{\, z\in D\,\mid (x,x)_z = (x,x)\,\}.
\eeq
This space is non-empty if and only if $(x,x)<0$, in which case $\sig(V_x) = (p,q-1)$. 
Then there is an isomorphism
$$D(V_x) \isoarrow D'_x, \qquad \zeta \mapsto [x,\zeta],$$
where, for an oriented $q-1$-plane $\zeta$ in $V_x$, $[x,\zeta]$ is the oriented $q$-plane in $V$ spanned by $x$ and $\zeta$ with 
the orientation induced by prepending $x$ to a properly oriented basis for $\zeta$. Note that $D'_x$ has codimension $p$ in $D$. 

In the case of main interest to us where $\sig(V) = (m-2,2)$, the $D_x$'s for positive vectors $x$ will have complex codimension $1$ in the hermitian 
symmetric space $D$ and the cycles $D'_x$ for negative vectors $x$ will be hyperbolic spaces of real dimension $m-2$, fixed points of
anti-holomorphic involutions of $D$.

\section{The main result}

We now consider the construction proposed in \cite{ABMP-II}, and so suppose that $\sig(V)= (m-2,2)$ and that $L$ is a lattice in $V$ such that 
with $L\subset L^\vee$. 
Our version of the ABMP data is a collection 
$\CC= \{C_1,\dots,C_N\}$  of $N$ vectors in $V$ satisfying the $N$-gon conditions (\ref{N-gon-c1}), (\ref{N-gon-c2}), and (\ref{N-gon-c3})
and hence defining an $N$-gon $\gamma(\CC)$ in $D^+$. 
Let $S=S(\CC)$ be an oriented $2$-chain in $D^+$ with boundary $\gamma(\CC)$ and let 
$$I^0(x,\CC) = \int_S \ph_{KM}^0(x)$$
be the integral of the $2$-form $\ph_{KM}^0(x)$ over $S$.  Since $\ph^0_{KM}(x)$ is closed, this integral does not depend on the choice of $S$. 
We will refer to it as the local theta integral. Our main result is the following. 



\begin{theo}\label{theoA} (i) The local theta integral is given explicitly by 
$$4\,I^0(x;\CC) = \wnat(\CC)+ \sum_{j=1}^N E_2(C_j,C_{j+1},x\sqrt{2}),$$
where, for a pair of negative vectors $c_1$ and $c_2$,  $E_2(c_1,c_2;x)$ is the generalized error function defined in \cite{ABMP}
and $\wnat(\CC)$ is defined by (\ref{def-wC-nat}).\hfb
(ii)  Let
$$\P(x;\CC):=\lim_{r\rightarrow\infty} I^0(r x, \CC).$$
Then
$$4\,\P(x;\CC)= \wnat(\CC)  + \sum_{j=1}^N  \sgn(x,C_j)\,\sgn(x,C_{j+1}) \ = \ \varepsilon(x;\mathcal C).$$
(iii) For $x\in V$, recall that $D_x=\{\ z\in D\mid x\in z^{\perp}\ \}$.
Then 
$$\P(x;\CC) \ne 0 \ \implies\ D_x\cap S\ne \emptyset.$$
In particular, for $x\ne 0$, $\P(x;\CC)\ne 0$ implies that $(x,x)>0$, since $D_x$ is non-empty. \hfb
(iv) Suppose that $(x,x)>0$. 
The cycle $D_x\subset D$ has complex codimension $1$ and the space $D^+ - D_x\cap D^+$ is a punctured disk bundle over the contractible 
space $D_x$. In particular, 
\beq\label{pi1-iso}
\pi_1(D^+ - D_x\cap D^+) =\Z, 
\eeq
canonically. Suppose that $x$ is regular with respect to $\CC$, i.e., $(x,C_j)\ne 0$ for all $j$.
Then, 
\begin{itemize}
\item[(a)] \vskip -8pt $4\,\P(x;\CC)\in \Z$,
 \item[(b)] the closed loop $\gamma(\CC)$ lies in $D^+ - D_x\cap D^+$, and, 
\item[(c)] under the identification (\ref{pi1-iso}),  its homotopy class is given by 
$4\,\P(x;\CC)$.
\end{itemize}
(v) The constant $\wnat(\CC)$ is independent of  the choice of negative vector $\vbold$. Moreover, $|\wnat(\CC)|\le N-2$ and $\wnat(\CC) \equiv -N \mod 4$. 
\end{theo}

\begin{rem} (a) For $x$ regular with respect to $\CC$, we can view $4\,\P(x;\CC)$ as the linking number of the cycles $\gamma(\CC)$ and $D_x$. \hfb
(b) The fact that $\wnat(\CC)$ does not depend on the choice of the negative vector seems to be a deeper result. At least we do not know a direct 
geometric proof. \hfb
(c) Once the independence of the choice of negative vector is known, the other conditions on $\wnat(\CC)$  given in (v) are immediate. Indeed, 
for a negative vector $\vbold$ with $(\vbold, C_1)=0$, at least $2$ terms in (\ref{def-wC-nat}) vanish so that $|\wnat(\CC)|\le N-2$. 
On the other hand, we can always take $\vbold$ so that all the inner products $(\vbold, C_j)\ne 0$ for all $j$, and so $-\wnat(\CC)$ is the sum of successive products of a collection of $N$ non-zero signs 
and hence $-\wnat(\CC)\equiv N\mod 4$. \hfb 
(d) In section~\ref{section-more-more}, we will show by examples that all values of $\wnat(\CC)$ subject to the conditions in (v) occur for some choice of $\CC$, cf. (\ref{all-w}). 
\end{rem}

\begin{cor} For $\mu\in L^\vee/L$ the series
$$\vartheta_\mu(\tau,\CC) = \sum_{x\in \mu+L} \varepsilon(x;\CC)\,\qbold^{Q(x)}$$
is termwise absolutely convergent. Its modular completion is given by 
$$\hat\vartheta_\mu(\tau,\CC) = \sum_{x\in \mu+L} \bigg(\  \wnat(\CC)+ \sum_{j=1}^N E_2(C_j,C_{j+1},x\sqrt{2})\ \bigg)\, \qbold^{Q(x)}$$
and
$$\hat\vartheta_\mu(\tau,\CC)  = 4 \int_{S}\theta(\tau,\ph_{KM}).$$
The weight $\frac{m}2$ modular transformation behavior of $\hat\vartheta_\mu(\tau,\CC)$ follows immediately from that of the theta form.
\end{cor} 
\begin{proof} The termwise absolute convergence follows from the argument in section~4.6 of \cite{kudla-int}. Recall that the key point is that there 
is a positive definite inner product $(\ ,\ )_S$ on $V$ such that $(x,x)_z\ge (x,x)_S$ for all $z\in S$ and all $x\in V$, cf. Lemma~4.11 of \cite{kudla-int}.  
On the other hand,  if $x\in V$ with $\P(x;\CC)\ne 0$, 
there exists a point $z_0\in S\cap D_x$, and at such a point $(x,x) = (x,x)_{z_0} \ge (x,x)_S$. Thus
$$\sum_{x\in \mu+L} |\P(x;\CC)|\,e^{-\pi v (x,x)} \le \frac12 N\sum_{x\in\mu+L} e^{-\pi v (x,x)_S}$$
converges, since $|\P(x;\CC)|$ is bounded by $\frac12 N$.  

The theta form
$$\theta_\mu(\tau,\ph_{KM}) = \sum_{x\in \mu+L} \ph_{KM}^0(v^{\frac12}x)\,q^{Q(x)}$$
is termwise integrable over the compact $2$-chain $S$. By (i) of Theorem~\ref{theoA}, this integral is the given expression
$\hat\vartheta_\mu(\tau,\CC)$.
\end{proof}

\begin{rem}\label{illegal-N-gons}
One can also consider a general collection $\CC= \{C_1,\dots,C_N\}$  of $N$ vectors in $V$ satisfying \eqref{N-gon-c1}) and \eqref{N-gon-c2} so that $z_j=[C_j,C_{j+1}]$ give rise to oriented negative two planes. By adjusting the signs of the $C_j$ one can always achieve that \eqref{N-gon-c3} holds for $j=1,\dots,N-1$. The case that \eqref{N-gon-c3} also holds for $j=N$ is discussed. If it doesn't, so $[-C_N,C_1]$ defines a point in the same component of $D$ as $z_1,\dots,z_{N-1}$ then our proof can be adjusted to give an analogous result. Namely, for such a collection define the invariant $\wnat(\CC)$ by 
\[
\wnat(\CC) :=  \sgn(\vbold,C_N)\,\sgn(\vbold,C_{1})-\sum_{j=1}^{N-1}  \sgn(\vbold,C_j)\,\sgn(\vbold,C_{j+1}).
\]
Then 
$$\hat\vartheta_\mu(\tau,\CC) := \sum_{x\in \mu+L} \bigg(\  \wnat(\CC) -  E_2(C_N,C_{1},x\sqrt{2}) +
\sum_{j=1}^{N-1} E_2(C_j,C_{j+1},x\sqrt{2})\ \bigg)\, \qbold^{Q(x)}$$
is modular of weight $m/2$ and the completion of the series obtained by replacing the higher error function $E_2$ by the corresponding signs. 
\end{rem}

\subsection{Comparison with \cite{ABMP-II}}We finish this section by explaining the relation to the setup of \cite{ABMP-II}.
As in section C of \cite{ABMP-II}, suppose that $N$ is even and consider a collection of vectors $\CC' =\{ C'_1,\dots,C'_N\}$
satisfying the conditions 
\begin{align}
(C'_j,C'_j)&<0\\
\nass
(C'_j,C'_j)(C'_{j+1},C'_{j+1}) - (C'_j,C'_{j+1})^2&>0\\
\nass
(C'_j,C'_j)(C'_{j-1},C'_{j+1}) - (C'_j,C'_{j-1})(C'_j,C'_{j+1}) &>0.
\end{align}
Here we have excluded the case of null vectors and have kept in mind the fact that our inner product is the negative of the 
one in \cite{ABMP-II}. Note that the left side of the third condition is $(C'_j,C'_j)(C'_{j-1\perp j},C'_{j+1\perp j})$, so that this condition 
is equivalent to (C.10) in \cite{ABMP-II}. The first two conditions are unchanged if the $C'_j$'s are changed by signs. If we let
\beq\label{ABMP-to-us}
\CC = \{C'_1,C'_2,-C'_3,-C'_4,C'_5, \dots, (-1)^{N/2-1}C'_N\}, 
\eeq
then the collection $\CC= \{C_1,\dots, C_N\}$ satisfies our conditions above, i.e., the sign in the third condition is reversed. 
Of course, the transformation (\ref{ABMP-to-us}) can be reversed to yield ABMP type collections from ours. 
Thus our results can be applied to prove the conjecture in 
section C of \cite{ABMP-II}, provided all of the vectors $C'_j$  are timelike, i.e., non-null. 
We have kernel
$$4\,\P(x;\CC) = w(\CC)+  \sum_{j=1}^N (-1)^{j-1} \sgn(x,C'_j)\,\sgn(x,C'_{j+1})$$
and 
$$w(\CC)=- \sum_{j=1}^N  (-1)^{j-1}\sgn(\vbold,C'_j)\,\sgn(\vbold,C'_{j+1})$$
for any negative vector $\vbold$. 

Note that our answer differs slightly from, and corrects, what was proposed in \cite{ABMP-II}.

The termwise absolute convergence goes over immediately to the case where the kernel is
$$\P(x,P;\CC) = P(x) \P(x;\CC)$$
where $P(x)$ is a function of at most polynomial growth. Also recall that it is shown in \cite{ABMP-II} that the case of 
odd $N$ follows from the case of $N$-even. It remains to remove the restriction that all the $C'_j$'s are non-null.

\section{Proofs}
We begin by considering what happens on the regular set
$$\text{Reg}(\CC) = V \setminus \bigsqcup_j V_j, \qquad  V_j = C_j^\perp,$$
and write $\gamma=\gamma(\CC)$. 
If $x$ is regular with respect to $\CC$, then $D_x \cap \gamma$ is empty, because every point of $\gamma$ has the 
form $[C_j, \star]$ and hence is not in $D_x$ since $(x,C_j)\ne 0$. 

Recall from Lemma~5.1 (ii) of \cite{FK-II}, that on the set $D-D_x$, we have
$$\ph^0_{KM}(x) = d\Psi_{KM}^0(x),$$
for an explicit $1$-form $\Psi_{KM}^0(x)$.
For $x$ regular with respect to $\CC$, so that $\gamma$ is contained in $D-D_x$,  let 
$$J^0(x;\CC):= \int_\gamma \Psi^0_{KM}(x).$$
Note that, from the formula for $\Psi_{KM}^0(x)$ on the set $D-D_x$, cf. (5.9) and (5.6) in section 5 of \cite{FK-II}, 
\beq\label{Psi-decay}
\lim_{r\rightarrow\infty}\Psi_{KM}^0(rx) = 0.
\eeq
and hence
\beq\label{lim-of-Psi}
\lim_{r\rightarrow\infty} J^0(r x, \CC) =0.
\eeq

For any $x$ (not necessarily regular), recall that we defined
\beq\label{coho-lim}
\P(x,\CC) = \lim_{r\rightarrow\infty} I^0(r x, \CC).
\eeq

\begin{lem}  Suppose that $D_x\cap S = \emptyset$. Then the limit (\ref{coho-lim}) exists and is equal to $0$. 
\end{lem} 
\begin{proof}
If $D_x\cap S$ is empty,  $\Psi_{KM}^0(x)$ is smooth on $\gamma$, 
$$I^0(x,\CC) = J^0(x,\CC),$$
by Stokes, and 
$$\lim_{r\rightarrow\infty} I^0(r x, \CC) =\lim_{r\rightarrow\infty} J^0(r x, \CC)=0.$$
\end{proof}
\begin{cor} If $\P(x,\CC) \ne 0$, then $D_x\cap S \ne \emptyset$. 
\end{cor}
Here the condition $\P(x,\CC) \ne 0$ includes the possibility that the limit does not exist. 

Now the analogue of the computations in \cite{kudla-int} and \cite{FK-II} yields the following. 
\begin{prop}\label{J0-formula}
For $x\in \text{\rm Reg}(\CC)$, 
$$J^0(x,\CC) = \int_\gamma \Psi_{KM}^0(x) = \frac14 \sum_{j=1}^N \bigg(\ E_2(C_j,C_{j+1},x\sqrt{2}) - \sgn(x,C_j)\,\sgn(x,C_{j+1})\ \bigg).$$
\end{prop}
\begin{proof} 
%
Using the formulas from \cite{FK-II}, we have  
\begin{align*}
\int_\gamma \Psi^0_{KM}(x) & = \sum_j \int_{\gamma_j} \Psi^0_{KM}(x)\\
\nass
{}&= -\sum_j \sqrt{2}\,(x,\und{C_j})\int_1^\infty e^{-2\pi t^2(x,\und{C_j})^2}\int_{\gamma_j} \ph_{KM}^{V_j,0}(t x_{\perp C_j}) \,dt\\
\nass
{}&= -\sum_j \sqrt{2}\,(x,\und{C_j})\int_1^\infty e^{-2\pi t^2(x,\und{C_j})^2} I^0(-C_{j-1\perp j},C_{j+1\perp j};t x_{\perp j}) \,dt\\
\nass
{}&= -\frac12\sum_j (\sqrt{2}x,\und{C}_j) \int_1^\infty e^{-2\pi t^2(x,\und{C_j})^2}  \big[\ E_1(C_{j+1\perp j},tx_{\perp j}\sqrt{2}) + E_1(C_{j-1\perp j},t x_{\perp j}\sqrt{2})\ \big]\,dt\\
\nass
{}&= \frac14\left[\,\sum_j E_2(C_j,C_{j+1};x\sqrt{2}) - \sum_j \sgn(C_j,x)\,\sgn(C_{j+1},x)\ \right].
\end{align*}
Here in the first step we use the relation of Corollary~6.3,
$$\kappa_j^*\Psi_{KM}^0(x) = - \sqrt{2}\,(x, \und{C}_j)\,\int_1^\infty e^{-2\pi t^2 (x, \und{C}_j)^2} \ph_{KM}^{V_j,0}(t x_{\perp C_j})\,dt,$$
where we write $\kappa_j:D(V_j)\lra D(V)$ for the standard embedding. 
This reduces the integral along $\gamma_j$ in $D'_{C_j}$ to the integral yielding 
Zwegers' identity.  In the last step, we use 
the recursion formula of Proposition~7.3 of \cite{FK-II}.  Recall,  that $\und{C} = |(C,C)|^{-\frac12} C$. 
\end{proof}

\begin{prop}\label{prop-identity} For $x\in \text{\rm Reg}(\CC)$, the limit (\ref{coho-lim}) defining $\P(x;\CC)$ exists and 
\begin{align*}
I^0(x;\CC)& = \P(x;\CC)+ J^0(x;\CC)\\
\nass
{}&= \P(x;\CC) + \frac14 \sum_{j=1}^N \bigg(\ E_2(C_j,C_{j+1},x\sqrt{2}) - \sgn(x,C_j)\,\sgn(x,C_{j+1})\ \bigg).
\end{align*}
\end{prop}

\begin{cor}  On the regular set $\text{\rm Reg}(\CC)$, the quantity 
$$\P(x;\CC) - \frac14 \sum_{j=1}^N  \sgn(x,C_j)\,\sgn(x,C_{j+1})$$
is the restriction 
of the continuous function 
$$I^0(x;\CC) -   \frac14 \sum_{j=1}^N E_2(C_j,C_{j+1},x\sqrt{2})$$
on $V$. 
\end{cor}

\begin{proof}[Proof of Proposition~\ref{prop-identity}]
The point is the following result, which reinvents the wheel. 

\begin{lem} For $x$ regular with respect to $\CC$ and for $r>0$, 
the difference 
$$I^0(rx;\CC) - J^0(rx;\CC), \qquad r>0,$$
is  independent of $r$. In particular, the limit (\ref{coho-lim}) exits
due to (\ref{lim-of-Psi}). 
\end{lem}
\begin{proof} 
From the definition, cf. section 5 of \cite{FK-II} and with the scaling by $r^{\frac12}$ used there, we have
\begin{align*}
\Ps_{KM}^0(r^{\frac12}x) &= -\int_{1}^\infty \psi^0_{KM}(t^{\frac12}r^{\frac12}x)\,t^{-1}\,dt\\
\nass
{}&=-\int_{r}^\infty \psi^0_{KM}(t^{\frac12}x)\,t^{-1}\,dt,
\end{align*}
where $\psi^0_{KM}(x)$ is the Schwartz $1$-form defined in (5.6) of \cite{FK-II}. 
Then
$$\frac{\d}{\d r} \Ps_{KM}^0(r^{\frac12}x) = r^{-1}\,\psi^0_{KM}(r^{\frac12}x).$$
Recall that, for $R(x,z) = |(\pr_z(x),\pr_z(x))|$, 
$$\psi^0(t^{\frac12}x) =\big(\,\text{form valued polynomial in $t^{\frac12}x$}\,\big) \cdot e^{-2\pi t R(x,z)},$$
so that the integral only makes sense when $R(x,z)>0$, i.e., on $D-D_x$, and the exponential decay (\ref{Psi-decay}) holds on this set. 
On the other hand, by (i) of Lemma~5.1, we have 
$$r \frac{\d}{\d r} \ph_{KM}^0(r^{\frac12}x)) = d\psi_{KM}^0(r^{\frac12}x).$$
Thus, 
\begin{align*}
&r\,\frac{\d}{\d r} \big(\ I^0(r^{\frac12}x;S) - J^0(r^{\frac12}x;S) \ \big)\\
\nass
{}&=\int_S r\frac{\d}{\d r}\ph_{KM}^0(r^{\frac12}x) - \int_{\d S} r\frac{\d}{\d r}\Psi_{KM}^0(r^{\frac12}x)\\
\nass
{}&=\int_S d\psi_{KM}^0(r^{\frac12}x)  - \int_{\d S} \psi^0(r^{\frac12}x)\\
\nass
{}&=0.\qedhere
\end{align*}
\end{proof}
This proves the proposition.
\end{proof}

A crucial property of the form $\ph_{KM}^0(x)$ is the following, whose proof will be given in the next section.  

\begin{prop}\label{prop-Z}  For $x$ regular with respect to $\CC$, $4\,\P(x,\CC)$ is in $\Z$.
\end{prop}



On the regular set, by Proposition~\ref{prop-identity}, we have the identity
\begin{align}\label{big-identity}
I^0(x;\CC) -\frac14 \sum_{j=1}^N E_2(C_j,C_{j+1},x\sqrt{2})  = \P(x;\CC) -\frac14 \sum_{j=1}^N  \sgn(x,C_j)\,\sgn(x,C_{j+1}).
\end{align}
The left hand side here is a continuous function on all of $V$, while, by Proposition~\ref{prop-Z},   
the right side lies in $\frac14\Z$ and hence is locally constant on the regular set. 
Thus the left hand side of (\ref{big-identity}) is constant on $V$. Computing the right hand side on any negative vector $\vbold$ in $\text{Reg}(\CC)$ 
and noting that $D_{\vbold}$ is empty so that $\P(\vbold;\CC)=0$, 
we obtain the constant 
$$-\frac14 \sum_{j=1}^N  \sgn(\vbold,C_j)\,\sgn(\vbold,C_{j+1}) = \frac14\,\wnat(\CC).$$
In particular, this value is independent of the choice of $\vbold$, as claimed.  

Thus  we have 
$$4\,I^0(x;\CC) =  \wnat(\CC)+\sum_{j=1}^N E_2(C_j,C_{j+1},x\sqrt{2}),$$
valid for all $x$ in $V$, 
as claimed in (i) of Theorem~\ref{theoA}.
Using this expression, we see that 
$$4\,\P(x;\CC) = \wnat(\CC) +  \sum_{j=1}^N  \sgn(x,C_j)\,\sgn(x,C_{j+1})$$
for all $x$ in $V$. 
This completes the proof of parts  (i), (ii) and (iii) of Theorem~\ref{theoA}. 
Part (iv) will be proved in the next section. 

%

\section{Proof of Proposition~\ref{prop-Z}}

\newcommand{\TT}{\mathcal T}

Let $D^+$ be the component of $D$ containing $\gamma=\gamma(\CC)$ and let $D_x^+ = D_x\cap D^+$. 
Let $\TT_\nu$ be the tube of radius $\nu$  around the cycle $D^+_x$ in $D^+$, and take $\nu$ small enough so that $\TT_\nu$ is bounded away from $\gamma(\CC)$. 
In particular, $\gamma$ lies in $D^+ - \TT_\nu$. 
Note that, for $\nu$ sufficiently small, $\TT_\nu-D_x^+$ is a punctured disk bundle over the contractible space $D_x^+$.   Taking a point $\zeta_0\in \gamma$ as base 
point, we have 
$$\pi_1(D^+-\TT_\nu, \zeta_0) = \pi_1(D^+-D^+_x,\zeta_0) \simeq \Z.$$
More explicitly,  let $D^\vee_{x,z_0}$ be the complementary cycle to $D^+_x$ through a point $z_0\in D^+_x$. Here recall that $D^+_x$ is the fixed point locus of the 
involution $\s_x:D^+\lra D^+$ induced by reflection in the subspace $V_x = x^\perp$. Then $D^\vee_{x,z_0}$ is the fixed point locus of the involution 
$\s_x\circ\theta_{z_0}$ where $\theta_{z_0}$ is the Cartan involution at $z_0$. Alternatively, $D^\vee_{x,z_0}$ is 
one component of the cycle associated to the subspace
$$V^\vee_{x,z_0} = z_0+\R x$$
of $V$ of signature $(1,2)$.  Since $D^+$ is the union of the cycles $D^\vee_{x,z_0}$ as  $z_0$ runs over $D_x^+$, we choose $z_0$ so that our base point $\zeta_0$ 
lies in $D^\vee_{x,z_0}$.  
We choose an isomorphism $D^\vee_{x,z_0} \simeq \H$ with $z_0\mapsto i$ so that $\TT_\nu\cap D^\vee_{x,z_0}$ maps to a small disk about  $i$.  The image in $\H$ of the 
base point $\zeta_0$ lies outside of this disk and we choose a disk $\Delta_i$ in $\H$ centered at $i$ and passing through this image. 
A generator of $\pi_1(D^+-\TT_\nu, \zeta_0)$ is then given by the image $\gamma_{x,z_0, \zeta_0}$ in $D^+-\TT_\nu$ of the  (counterclockwise) boundary of $\Delta_i$.

Our piecewise smooth loop $\gamma(\CC)$ in $D^+-\TT_\nu$ passes through $\zeta_0$ and is homotopic to an integer multiple $k$ of $\gamma_{x,z_0,\zeta_0}$. Moreover the homotopy can be realized 
as a piecewise smooth oriented $2$-chain $\hbold$ in $D-\TT_\nu$, i.e., 
$$\d \hbold = \gamma(\CC) - k\cdot \gamma_{x,z_0,\zeta_0}.$$ Let $S_{x,z_0}$ be the image of $\Delta_i$ in $D^\vee_{x,z_0} \subset D^+$. 
We may then take 
$$S = \hbold +k\cdot S_{x,z_0},$$
so that 
$$I^0(x;\CC) = \int_{\hbold }\ph^0_{KM}(x) + k \cdot \int_{S_{x,z_0}} \ph^0_{KM}(x).$$
Now $R(x,z)$ is uniformly bounded away from $0$ on $D^+-\TT_\nu$, and so 
$$\lim_{r\rightarrow\infty}\int_{\hbold }\ph^0_{KM}(rx) =0.$$
{\bf Claim:} On the other hand, 
$$\lim_{r\rightarrow\infty} \int_{S_{x,z_0}} \ph^0_{KM}(rx)  = \int_{D^\vee_{x,z_0}}\ph^0_{KM}(x) = \frac14,$$
by the Thom form property of $\ph^0_{KM}(x)$, \cite{KM.I}, \cite{KM.II}. 
Thus 
\beq\label{homotopy-mult}
4\,\P(x;\CC) = k.
\eeq
To justify the limit, note that 
$$\int_{S_{x,z_0}} \ph^0_{KM}(rx) = \int_{D^\vee_{x,z_0}}\ph^0_{KM}(rx) - \int_{D^\vee_{x,z_0}-S_{x,z_0}}\ph^0_{KM}(rx).$$
But $R(x,z)$ is bounded below by some $R_0$, uniformly for $z\in D^+-\TT_\nu$ and the second integral here 
decays at least like $C \exp(-\pi r^2 \frac12 R_0)$ as $r$ goes to infinity. 
On the other hand, by Proposition~6.2 of \cite{KM.II}, we have
$$\int_{D^\vee_{x,z_0}}\ph^0_{KM}(rx) =\frac14.$$
Note that the integral in \cite{KM.II} is taken over both connected components, whereas our domain of integration $D^\vee_{x,z_0}$ 
is connected. 

{\bf Remark:}  The formula (\ref{homotopy-mult}) gives a nice explanation of the integer $4\,\P(x;\CC)$.

\section{Examples:  Signature $(1,2)$}\label{section-more-more}

We now consider the case where  $(V,Q)$ is the 
quadratic space of traceless $2 \times 2$ matrices with the quadratic form $Q(X) = \det(X)$. The corresponding bilinear form $(X,Y) = -\tr(XY)$ has signature is $(1,2)$. 
There are identifications
$$\H^\pm \isoarrow \hh^\pm \isoarrow D=D(V)$$
where $\H^\pm = \C-\R$ and
$\hh^\pm = \{ \, X\in V\mid  Q(X)=1\,\}$ given by  
$$z = x+iy \mapsto X(z) = y^{-1}\bpm -x&|z|^2\\ -1&x\epm, \qquad X \mapsto X^\perp = \zbold ,$$
where a basis $u_0$, $u_1$ for the negative $2$-plane $X^\perp$ is properly oriented if  $X\wedge u_0\wedge u_1 \in \bigwedge^3(V)^+$.  Here we orient $V$ by choosing
the ordered orthogonal basis 
\[
e_1= \zxz{}{1}{-1}{}, \qquad  e_2= \zxz{1}{}{}{-1},\qquad  e_3= \zxz{}{-1}{-1}{}
\]
and taking $\bigwedge^3(V)^+ = \R_{>0}\cdot e_1\wedge e_2\wedge e_3$. 
The components of $\hh^\pm$ are determined by $X\in \hh^\pm$ if $\sgn(X,e_1)= \pm1.$
The inverse of the second isomorphism is given as follows. If $u_0$, $u_1$ is an orthogonal basis for the negative $2$-plane $\zbold$, take a vector $X\in \zbold^\perp$
with $Q(X)=1$ and such that $X\wedge u_0\wedge u_1\in \bigwedge^3(V)^+$. 
Explicitly, if $u_0= u_{10}e_1 +  u_{20}e_2 + u_{30}e_3$ and $u_1=u_{11}e_1 +  u_{21}e_2 + u_{31}e_3$, the vector $X$ is given by the hyperbolic cross product
$$X = u_0\cross u_1 = (u_{20}u_{31}-u_{30}u_{21})\, e_1 + 
 (u_{10}u_{31}-u_{30}u_{11}) \,e_2 -  (u_{10}u_{21}-u_{20}u_{11})\, e_3.$$
 (Note the signs for the second and third component).


Vectors $X\in V$ with $Q(X)\ne0$ define cycles in $D$ by 
\begin{align*}
D_X &= \{\, \zbold \in D\mid X\in \zbold^{\perp}\,\},&&\text{if $Q(X)>0$,}\\
\noalign{\vskip -10pt and}
c_X &= \{\, \zbold\in D\mid X\in \zbold \,\},&&\text{if $Q(X)<0$},
\end{align*}
so that $D_X$ is a union of two points $D^\pm_X$ and $c_X$ is a union of two geodesic $c_X^\pm$, one in each component of $D$. 
Note that $D_X$ determines the positive line $\R\cdot X$ when $Q(X)>0$, and $c_X$ determines the negative line $\R\cdot X$ when $Q(X)<0$. 
In the later case $\R\cdot X$ is the intersection of any two distinct points on $c_X$, viewed as negative $2$-planes.  

These cycles 
can be described in $\H^\pm$ in a familiar classical way.  
For 
\[
X = \zxz{b}{2c}{-2a}{-b}=: [a,b,c],\qquad Q(X) = b^2-4ac = d,
\]
we have 
\begin{align*}
(X,X(z))=  \frac{2}{y} (a|z|^2+bx+c) \qquad \text{and} \qquad 
R(X,z)   = - (\pr_{\zbold}(X),\pr_{\zbold}(X)) = \frac{2}{y^2}|az^2+bz+c|^2.
\end{align*}
Then 
\begin{align*}
D_X &= \frac{-b}{2a} \pm \frac{i\sqrt{d}}{2a}, &&\text{if $Q(X)>0$}\\
\noalign{\vskip -10pt and}
c_X &= \{\,z \in \H^\pm; \; a|z|^2+b\Re(z)+c=0\,\},&&\text{if $Q(X)<0$}.
\end{align*}
%
Note that $c_X^+$
 separates $D^+\simeq \H^+$ 
 into two 
 components. We orient the cycle by asserting that the component given by $(X(z),X)>0$ lies to the left of the geodesic and the one given by $(X(z),X)<0$ to the right. Hence for $X=[a,b,c]$ with $a \ne 0$ as above, $c_X^+$ is a semi-circle in $D^+$ oriented clockwise when $a>0$ and the vertical geodesic arising from $X = [0,b,c]$ gives for $b>0$ the oriented half line $(-\tfrac{c}{b} + i \infty, -\tfrac{c}{b})$. 
 
%
%

Given two points $z_1$, $z_2$ in the same component of $\H^\pm$, say $\H^+$, the unique geodesic $c_{12}$ containing both points arises from the negative vector 
\[
Y(z_1,z_2):=X(z_1) \times X(z_2) = \frac{1}{y_1y_2} \zxz{\frac12(|z_1|^2-|z_2|^2)}{x_1|z_2|^2-x_2|z_1|^2}{x_1-x_2}{\frac12(|z_2|^2-|z_1|^2)}.
\]
Indeed, this negative vector is orthogonal to both $X(z_1)$ and $X(z_2)$ and hence lies in $\zbold_1$ and in $\zbold_2$,  so that $\zbold_1$ and $\zbold_2$ both lie on $c_{12}$. 
Moreover, $c_{12}$ is oriented so that it runs from $z_1$ to  $z_2$.  By $\SL_2(\R)$ equivariance, it suffices to check for $z_1=i$ and $z_2= ir$, $r>0$ and this case is immediate. 

Now suppose that we have three distinct points $z_1$, $z_2$ and $z_3$ in $\H^+$ and hence geodesic arcs $c_{12}$ and $c_{23}$ associated to $Y_{12} = X(z_1)\cross X(z_2)$ 
and $Y_{23}= X(z_2)\cross X(z_3)$. The negative vectors $Y_{12}$ and $Y_{23}$ both lie in $\zbold_2\in D^+$ so that it is natural to ask if $[Y_{12},Y_{23}]$ is a properly 
oriented basis. 
\begin{lem}\label{left-right-turn}  The cross product of $Y_{12}$ and $Y_{23}$ is given by
\beq\label{double-cross}
\big(X(z_1)\cross X(z_2)\big)\cross \big(X(z_2)\cross X(z_3)\big) = \a(z_1,z_2,z_3)\, X(z_2).
\eeq
where
\beq\label{a-invar}
\a(z_1,z_2,z_3) = \frac{x_1(|z_2|^2-|z_3|^2)+x_2(|z_3|^2-|z_1|^2)+x_3(|z_1|^2-|z_2|^2)}{2y_1y_2y_3}.\eeq
In particular, 
\begin{equation}
\zbold_2 = [X(z_1)\cross X(z_2), X(z_2)\cross X(z_3)] \qquad \iff \qquad \a(z_1,z_2,z_3) >0.
\eeq
Moreover, $\a(z_1,z_2,z_3)>0$  if and only if the geodesics associated to $X(z_1) \times X(z_2)$ and $X(z_2) \times X(z_3)$ perform a left turn at $z_2$. 
\end{lem}

\begin{rem}
For example, 
 for $z_1=i$, $z_2 = i r$ and $z_3= x+iy$, the expression in the lemma is $x(1-r^2)$ so that, if $r>1$, the expression is positive 
for $x<0$ (left turn) and negative for $x>0$ (right turn). 
\end{rem}

Here is a recipe for recovering $N$-gon data $\CC= \{ \,C_1, \dots, C_N\,\}$ from a geodesic $N$-gon in $\H\simeq D^+$ with vertices $\mathcal Z= \{\, z_1, \dots, z_N\,\}$.   
The geodesic arc in $D^+$ containing $\zbold_{j-1}=[C_{j-1},C_j]$ and $\zbold_j=[C_j, C_{j+1}]$ is $c_{C_j}^+$, while that containing $\zbold_j$ and $\zbold_{j+1}=[C_{j+1},C_{j+2}]$ 
is $c^+_{C_{j+1}}$.  These arcs must be distinct since otherwise the vectors $C_j$ and $C_{j+1}$ would be collinear, whereas they span the $2$-plane $\zbold_j$. 
Also, if $z_j\in \H$ corresponds to $\zbold_j\in D^+$, then $C_j$ must be proportional to $Y_{j-1 j}=X(z_{j-1})\cross X(z_{j})$ since they determine the same geodesic arc. As observed earlier, the geometry does not change if
the vectors in $\CC$ are scaled independently by positive constants. Up to such a scaling, we may then write $C_j = \e_j Y_{j-1 j}$ for some signs $\e_j = \pm1$. 
Also notice that at each vertex $z_j$, either a left or a right turn in made, since successive arcs are distinct. Let $\tau_j= +1$ for a left turn and $-1$ for a right turn. 
It then follows from Lemma~\ref{left-right-turn} that
$$\zbold_j = [Y_{j-1j}, \tau_j Y_{j j+1}] = [\e_j C_j, \tau_j\e_{j+1}C_{j+1}] = [C_j,C_{j+1}],$$
so that $\e_{j+1}= \e_j \tau_j$. Also note that the geometry does not change if all $C_j$ are scaled by $-1$.
Thus, we may assume that $\e_1=1$ and conclude that, for $j\ge 2$, 
\beq\label{turning-ep}
\e_j = \prod_{1\le i\le j-1} \tau_i.
\eeq
Note that the crucial fact is that the chain of geodesics must close up, so that 
$$\zbold_N= [C_N,C_1] = [Y_{N-1 N},\tau_N Y_{N,1}]= [\e_N C_N,\tau_N \e_1 C_1].$$
Since $\e_1=1$, this amounts to the total turning condition
\beq\label{total-turning}
\prod_{j=1}^N \tau_j=1.
\eeq

Thus we have the following recovery recipe. 
\begin{prop} (i)  Let $\mathcal Z=\{ \, z_1, \dots, z_N\,\}$ be the set of vertices of an $N$-gon in $\H\simeq D^+$ defined by a collection $\CC= \{\, C_1, \dots, C_N\,\}$. Then the geodesic arcs joining successive vertices are distinct and, up to admissible scaling,  
$C_j=\e_j X(z_{j-1})\cross X(z_j)$ where $\e_j$ is determined by the turning signs $\tau_j$ via (\ref{turning-ep}).\hfb 
(ii)
Suppose that  $\mathcal Z=\{ \,z_1,\dots, z_N\,\}$ is an ordered set of $N$ points in $\H$ with no three consecutive points collinear. Let 
$$C_j = \e_j X(z_{j-1})\cross X(z_j),$$
where $\e_j$ is defined by (\ref{turning-ep}) such that the total turning condition (\ref{total-turning}) is satisfied. 
Then the collection 
$$\CC = \{\, C_1, \dots, C_N\,\}$$
satisfies the $N$-gon conditions and the associated $N$-gon has vertices $\z_j$. 
\end{prop}
\begin{proof} It remains to check that the collection $\CC$ in part (ii) satisfies the $N$-gon conditions. The cross product is equivariant for the action of $\SL_2(\R)$, i.e., 
$(g\cdot A)\cross (g\cdot B) = g\cdot (A\cross B)$. Since 
$$\a(i, ir, x+iy) = (2ry)^{-1}\,x\,(1-r^2),$$
the requirements that the $z_j$'s give rise to distinct successive geodesics imply that $\a(z_{j-1},z_j,z_{j+1})\ne 0$ for all $j$ and hence
$C_j\cross C_{j+1}$ is a nonzero multiple of $X(z_j)$. Thus $[C_j,C_{j+1}]$ is a negative $2$-plane. 
We must still check the condition 
\beq\label{check-3}
(C_j,C_j)(C_{j-1},C_{j+1})  - (C_j,C_{j-1})(C_j,C_{j+1})<0,
\eeq
where we note that 
$$\e_{j-1}\e_{j+1} = \tau_{j-1}\tau_j.$$
By equivariance, it suffices to consider the sequence of $4$ points 
$$\{z_{j-2},z_{j-1}, z_{j}, z_{j+1}\} = \{\, i, ir, x+iy, x'+iy'\,\}$$ 
with $r>1$. The turning signs are then 
\begin{align*} 
\tau_{j-1} &= \sgn(\a(i,ir,x+iy)) = \sgn(x (1-r^2))\\
\nass
\tau_j &= \sgn(\a(ir, x+iy,x'+iy')) =\sgn( r^2(x'-x) + x |z'|^2 - x' |z|^2) .
\end{align*}
The quantity on the left side of (\ref{check-3}) is then $\tau_{j-1}\tau_j$ times 
$$(2 r^3 y^3 y')^{-1} (1-r^2)\, x\, (r^2+|z|^2) (r^2 (x-x') + x' |z|^2 - x |z'|^2),$$
so that condition (\ref{check-3}) holds!
\end{proof} 

\begin{rem} Note that the condition (\ref{total-turning}) in the characterization of the $N$-gons arising from collections $\CC$ precludes configurations with an odd number of right turns. For example, a triangle with edges oriented clockwise cannot arise. However, if the $N$-gon defined by $\mathcal Z=\{ \, z_1, \dots, z_N\,\}$ has an even number of left turns one can simply swap the orientation of the $N$-gon to obtain an `admissible' collection $\CC$, and the indefinite theta series to the surface $S$ associated to $\mathcal{Z}$ will be the negative of the one associated to $\CC$. This takes care of $N$ odd. On the other hand, for $N$ even it is easy to construct an $N$-gon with $(N-1)$ left turns and one right turn. In that case Remark~\ref{illegal-N-gons} applies whose formula also can be easily shown directly as follows. The right turn occurs at $z_N$ then use another negative vector $C_N'$ to geodesically cut-off the right turn close to $Z_N$ to bisect $S$ into a (clockwise oriented) triangle formed by $\{\,C_N,C_1,C_N'\}$and an `admissible' $(N+1)$-gon obtained from $\{\, C_1, \dots, C_N\,C_N'\}$. Then combining the two indefinite theta series gives the one given in Remark~\ref{illegal-N-gons}.
 \end{rem}

Next we consider the invariant 
$$w(\CC) = -\sum_j \sgn((e_2,C_j))\,\sgn((e_2, C_{j+1})).$$
Note that 
$$( e_2, Y(z_1,z_2)) = (y_1 y_2)^{-1} (|z_2|^2 - |z_1|^2),$$
so that 
$$ \sgn((e_2, C_j)) = \e_j \,\sgn(|z_j|^2-|z_{j-1}|^2).$$
Since $\e_j \e_{j+1} = \tau_j$, 
\beq\label{one-sign}
\sgn((e_2, C_j)) \,\sgn((e_2, C_{j+1})) = \tau_j \,\sgn(|z_j|^2-|z_{j-1}|^2)\,\sgn(|z_{j+1}|^2-|z_{j}|^2).
\eeq
Thus, the contribution of this term is determined by the turn made at $z_j$ and whether the incoming and outgoing edges at this point
are inward or outward bound with respect to the origin in $\C$, information that can be read off immediately from the configuration $\mathcal Z=\{\,z_1,\dots,z_N\,\}$. 


Here are some examples. Let $z_r = r+i$, for $r\in \Z$. 
For the chain
$$\mathcal Z= \{\,z_0, z_1, \dots, z_{N-1}\,\},$$
only left turns occur, i.e., 
$\tau_j=1$ for all $j$, and for $1\le j \le N-2$ the incident edges are both outbound, so these vertices contribute $2-N$ to the sum.
At the vertices $z_0$ and $z_{N-1}$ the incident edges are have opposite directions and so these two vertices contribute an additional $+2$, 
and $w(\CC) = 4-N$. 

Next let $z'_r = r+i T$ for $T\gg1$ and $r\in \Z$, and consider the chain
$$\mathcal Z = \{\, z_0, \dots, z_{k-1}, z'_{\ell-1}, z'_{\ell-2},\dots, z'_0\, \}, \qquad N=k+\ell.$$
Left turns are made at the vertices $z_0, \dots, z_{k-1}, z'_{\ell-1}$ and $z'_0$, and right turns are made at the remaining vertices $z'_1,\dots, z'_{\ell-2}$.  
Thus we require\footnote{Formally, for $\ell=0$ we obtain the previous example.} that $\ell\ge 2$ is even. At the vertices $z_0$ and $z'_{\ell-1}$, the incident edges are have opposite directions, so these contribute $+2$. 
At the vertices $z_1, \dots, z_{k-1}$, the incident edges are both outbound resulting in a net contribution of $1-k$ 
to $w(\CC)$. At the vertex $z'_0$ the incident edges are both inbound yielding a contribution of $-1$. Finally, at the remaining vertices, $z'_1, \dots, z'_{\ell-2}$
where right turns occur, the incident edges are both inbound, so these contribute $\ell-2$. Hence
$w(\CC) =\ell-k = N-2k.$  Since $k\ge 1$ and $\ell\ge2$  resulting range of values for $w(\CC)$ is 
\begin{align}\label{all-w}
\{4-N, 8-N, \dots, N-8, N-4\}&\qquad\text{for $N$ even,}\\
\nass
\{4-N, 8-N, \dots, N-6, N-2\}&\qquad\text{for $N$ odd.}\notag
\end{align}

\begin{rem} Any quadratic space $V$ of signature $(m-2,2)$ for $m\ge 3$ contains a subspace of signature $(1,2)$ and the $N$-gons 
just described can be realized in $D(V)^+$. Thus  all of the values in (\ref{all-w}) occur in the general case. 
\end{rem}



%

We close the section with some additional explicit examples. 


Consider for $T>1$
\begin{align*}
C_1 &= \zxz{-1}{1}{}{1} = \frac12 e_1 -e_2 - \frac12 e_3, \qquad 
C_{3} = \zxz{1}{1}{}{-1} = \frac12 e_1 + e_2 - \frac12 e_3, \\
C_{2} &=\zxz{}{T^2+1/4}{1}{} = \frac{T^2-\tfrac34}2 e_1- \frac{T^2+\tfrac{5}{4}}2 e_3,   \qquad
C_4 = \zxz{}{-1}{-1}{} = e_3.
\end{align*}
These vectors satisfy the $N$-gon condition, and the corresponding vertices in $\H^+$ are 
\begin{align*}
z_{1}= \frac12 + Ti,
 \quad 
 z_{2}=- \frac12 + Ti,
 \quad
 z_{3}= -\frac12 + \frac{\sqrt{3}}{2}i,
\quad 
z_{4}= -\frac12 + \frac{\sqrt{3}}{2}i. 
 \end{align*}
Note that $c^+_{C_1}$ is the vertical geodesic $x=\tfrac12$ going up, $c^+_{X_3}$ the one $x=-\tfrac12$ going down, $c^+_{C_2}$ is the half circle $|z|^2=T^2+1/4$ going anti-clockwise, and $c^+_{4}$ the half circle $|z|=1$ going clockwise. Hence $\gamma(\mathcal{C})$ is the boundary of the standard fundamental domain for $\SL_2(\Z)$ cut-off at hight $T$ (and connected by the geodesic arc segment), and we can take $S_T(\CC)$ to be region enclosed by  $\gamma(\mathcal{C})$. Further note $w(\CC)=0$ and hence 
\begin{align*}
I(D_X, S_T(\CC)):= \frac14\varepsilon(X;\mathcal C) = \frac14 \sum_{j=1}^4 \sgn(X,C_j)\sgn(X,C_{j+1})
\end{align*}
gives the intersection number $I(D_X, S_T(\CC))$ of the $0$-cycle $D_X$ with the surface $S_T(\mathcal{C})$, that is, $I(D_X, S_T(\CC))=1$ if $D_X$ lies in the interior of $S_T(\mathcal{C})$ and $I(D_X, S_T(\CC))=0$ if $D_X$ lies outside $S_T(\mathcal{C})$. 

We now fix the lattice $L=\{ [a,b,c]; \, a,b,c \in \Z\}$ which is stabilized by $\Gamma=\SL_2(\Z)$. Then 
\[
\int_{S_T(\CC)} \theta(\tau,\varphi_{KM},L)= \frac14  \sum_{j=1}^4  \sum_{X \in L} E_2(C_j,C_{j+1},X\sqrt{2})\, \bf{q}^{Q(x)} 
\]
is the modular completion of the 
\[
\sum_{X \in L; Q(X)\geq 0} I(D_X, S_T(\CC)){\bf q}^{Q(X)}.
\]
Note for fixed $n>0$, the quantity $ \sum_{\substack{X \in L \\Q(X)=n}} I(D_X, S_T(\CC))$ exactly counts twice the number of CM points of discriminant $-n$ in the standard fundamental domain $\mathcal{F}$ for $\SL_2(\Z)$ below the geodesic defined by $C_2$ (the factor $2$ arises from $x$ and $-x$ giving rise to the same CM point.) Thus we can view the theta integral as the modular completion for the generating series of ``truncated" class numbers. 

It is natural to take the limit $T \to \infty$. In fact, the first author considered exactly the resulting theta integral over the non-compact domain $\mathcal{F}$ in his PhD thesis written under the supervision of the second author. We have 

\begin{theo}\cite{Funke-diss}[Zagier's weight $3/2$ Eisenstein series.]\label{th:Zagier}
The theta series $ \theta(\tau,\varphi_{KM},L)$ is exponentially decaying as a differential form at all rational cusps of $\H$ (but not term wise), and the theta integral $\int_{\mathcal{F}} \theta(\tau,\varphi_{KM},L)$ is equal to Zagier's weight $3/2$ Eisenstein series \cite{Zagier, HZ}, i.e.,
\[
2 \int_{\mathcal{F}} \theta(\tau,\varphi_{KM},L) = \sum_{n \geq 0} H(n) q^n + \frac{v^{-1/2}}{16 \pi} \sum_{n \in \Z} \beta(4\pi n^2 v) q^{-n^2}. 
\]
Here $H(n)$ denotes the Kronecker Hurwitz class number (with $H(0)=-1/12$) and $\beta(s)= \int_1^{\infty} e^{-st}t^{-3/2} dt$.
\end{theo}

 \begin{rem}\label{null-vector-remark}
 Since $\lim_{T \to \infty} \tfrac1{T^2} C_2 = \kzxz{0}{1}{0}{0}$ we can view the limit process as replacing $C_2$ with the null vector $u:=\kzxz{0}{1}{0}{0}$.  In fact, \cite{ABMP-II} in their treatment allow the $C's$ to be rationally null as well. It is hence suggestive to replace $C_2$ with $u$ and perform the same analysis as before. However, because of convergence issues one needs to proceed carefully. In fact, the non-holomorphic tail in Zagier's series exactly arises from the failure of term wise convergence. 

Nevertheless, in \cite{FM-boundary} it is shown that the theta series $ \theta(\tau,\varphi_{KM})$ extends to a differential form on the Borel-Serre enlargement $\overline{D}$ of the symmetric space $D$ (which eg. for the upper half plane adds a real line to each rational cusp). Then a pair of non-zero non-positive vectors $C_j,C_{j+1}$, rational if null, will define a point $z_j$ in $\overline{D}$. If the collection $\CC$ satisfies the (analogous) $N$-gon conditions then one can obtain a closed loop $\gamma(\CC)$ in $\overline{D}$ as before and then a region $S(\CC)$ over which one can integrate the theta series. We will revisit this case in the near future. Note that for signature $(p,1)$ we discussed the null case in \cite{FK-I}. 
  \end{rem}
 

Now consider 
\begin{align*}
C'_1 &= \zxz{-3/2}{-5/2}{-1}{3/2} = -\frac34 e_1- \frac32 e_2 - \frac74 e_3, \qquad 
&C'_{3} = \zxz{3/2}{-5/2}{-1}{-3/2} = -\frac34 e_1 + \frac32 e_2 - \frac74 e_3,  \\
C'_{2} &=\zxz{}{4}{1}{} =  \frac32 e_1- \frac52 e_3,\qquad 
&C'_4= \zxz{}{1}{1}{} = -e_3.
\end{align*}
The associated geodesics intersect at the same four points as above (with $T=\sqrt{15}/2$), however in a different configuration:
\begin{align*}
z_{1}= \frac12 + \frac{\sqrt{15}}{2}i,
\quad
z_{2}=- \frac12 + \frac{\sqrt{15}}{2}i,
\quad 
z_{3}= \frac12 + \frac{\sqrt{3}}{2}i,
\quad 
z_{4}= -\frac12 + \frac{\sqrt{3}}{2}i.
\end{align*}
Now $\gamma_1$ is the geodesic ``diagonal" connecting the lower left vertex $-\frac12 + \frac{\sqrt{3}}{2}i$ with the upper right one $\frac12 + \frac{\sqrt{15}}{2}i$  (in this direction) and $\gamma_3$ is the other diagonal connecting $-\frac12 + \frac{\sqrt{15}}{2}i$ with $\frac12 + \frac{\sqrt{3}}{2}i$. Of course $\gamma_1$ and $\gamma_{3}$ intersect  in the interior of the square, but this is not an issue. The geodesic $\gamma_2$ is the half circle $|z|=2$ oriented counter-clockwise while $\gamma_4$ connects $ \frac12 + \frac{\sqrt{3}}{2}i$ with $-\frac12 + \frac{\sqrt{3}}{2}i$ on the unit circle. Note that while the geodesic associated to $C_4'$ is oriented {\it clockwise}, the geodesic arc segment $\gamma_4$ corresponding to $C_4'$ is oriented {\it counter-clockwise}. The reason is that $[C_3',C_4']$ and $[C_4',C_1']$ define points on the upper half plane (and not $[C'_3,-C_4']$ and $[-C_4',C_1']$).

The resulting curve $\gamma(\CC')$ defines a `butterfly', and the corresponding surface $S(\CC')$ is in fact the union of two oriented triangles,
\[
S(\CC)=S(C_1',C_2',C_3') - S(-C_3',-C_1',C_4').
\]
We has as before $w(\CC')=0$ and hence 
\[
I^0(x,\CC')= \frac14 \sum_{j=1}^4 E_2(C'_j,C'_{j+1};x).
\]
Considering the two triangles separately we obtain (now one has $w=1$ each) one obtain the same. 
Note that the formula 
\[
I(D_x,S(\mathcal{C}'))= \frac14\sum_{j=1}^4 \sgn(x,C_j)\sgn(x,C_{j+1})
\]
is still valid: It gives $1$ in the upper triangle, $-1$ in the lower (opposite orientation!), and $0$ outside both.

\section{A general result} 

The methods developed above can be applied in a rather more general situation. 

Suppose that $V= L\tt_\Z\R$ has signature $(p,q)$ with $q\ge 2$ and let $D$ be the space of oriented negative $q$-planes in $V$. 
Fix a base point $z_0\in D^+$ where $D^+$ is one component of $D$. Let $I=[0,1]$.
Suppose that 
$$\gamma: (I^{q-1},\d I^{q-1}) \lra (D,z_0)$$
is piecewise smooth map, where the notation means that $\gamma:I^{q-1}\lra D^+$ is a piecewise smooth map with $\gamma(\d I^{q-1}) = z_0$. 
Since $D^+$ is contractible, there exists a homotopy
$$\rho: I^q\lra D^+$$
such that
$$ \rho\mid I^{q-1}\times 0 = \gamma, \qquad \rho(I^{q-1}\times\{1\}) = z_0.$$
We can assume that $\rho$ is also piecewise smooth and thus, viewing $\rho$ as a $q$-chain, we can define
$$I^0(x;\gamma) = \int_{I^q} \rho^*\ph^0_{KM}(x).$$
Since $\ph^0_{KM}(x)$ is closed, this integral does not depend on the choice of $\rho$. It defines a continuous function of $x\in V$. 

Let $D_x$ be defined as in (\ref{def-Dx}). For $x\ne0$, $D_x$ is empty if $(x,x)\le 0$ and is a total geodesic subsymmetric space 
of codimension $q$ if $(x,x)>0$. When  $(x,x)>0$,  there is a fibration $\pr:D \lra D_x$ where, for $z_1\in D_1$, the fiber $\pr^{-1}(z_1) = D^\vee_{x,z_1}$
is the subsymmetric space associated to the subspace $V^\vee_{x,z_1} = z_1 + \R x$ of $V$ of signature $(1,q)$. The space $D^+-D^+_x$ is then 
a fiber bundle over $D^+_x$ with fibers $D^\vee_{x,z_1} - \{z_1\}$. 
Also
$$\pi_{q-1}(D^+-D^+_x, z_0) = \Z$$
with generator given by a (suitably oriented) sphere $S_{x,z_1}$  in $D^\vee_{x,z_1}-\{z_1\}$ for $z_1 = \pr(z_0)\in D_x^+$. 

We say that $x$ is regular with respect to $\gamma$ is $D_x\cap \gamma(I^{q-1})$ is empty. 
For $x\ne0$, this is always the case when $(x,x)\le 0$. If $(x,x)>0$ and $x$ is regular with respect to $\gamma$, then $z_0\notin D_x$ and 
$$\gamma: (I^{q-1},\d I^{q-1}) \lra (D^+-D^+_x,z_0).$$
Then $\gamma$ is homotopic to a multiple $k \cdot S_{x,z_1}$ of the generator for $\pi_{q-1}(D^+-D^+_x,z_0)$. 
We may suppose that there is a piecewise smooth $q$-chain $\hbold$ in $D^+-D^+_x$ such that 
$$\d \hbold = \gamma - k\cdot S_{x,z_1}.$$
Thus, if $B_{x,z_1}$ is the closed disk in $D^\vee_{x,z_1}$ with boundary $S_{x,z_1}$, we can take the $q$-chain
$$S = \hbold + k\cdot B_{x,z_1}$$
in the definition of $I^0(x;\gamma)$. 

Recall that there is a $q-1$-form $\Psi^0_{KM}(x)$ such that 
$$\ph^0_{KM}(x) = d\Psi^0_{KM}(x)$$
and the properties of $\Psi^0_{KM}(x)$ used in the arguments above hold in general. 

Let 
$$\P(x;\gamma) = \lim_{r\rightarrow\infty} I^0(rx;\gamma),$$
and if $x$ is regular with respect to $\gamma$, let
$$J^0(x;\gamma) = \int_{I^{q-1}}\gamma^*\Psi^0_{KM}(x).$$
Note that 
$$\lim_{r\rightarrow\infty} J^0(rx;\gamma)=0.$$
Also, we have
\begin{align*}
\P(x;\gamma) &= \lim_{r\rightarrow\infty} I^0(rx;\gamma)\\
\nass
{}&=\lim_{r\rightarrow\infty}\bigg(\ \int_{\hbold} \ph^0_{KM}(rx) + k\cdot \int_{B_{x,z_1}} \ph^0_{KM}(rx)\ \bigg)\\
\nass
{}&= k\cdot \int_{D^\vee_{x,z_1}} \ph^0_{KM}(rx)\\
\nass
{}&=k\cdot \varepsilon(x)\, 2^{-\frac{q}2-1}
\end{align*}
by Proposition~6.2 of \cite{KM.II}, where the sign $\varepsilon(x)$ depends on choices of orientations which we will not make explicit here. When $q$ is even, $\varepsilon(x)=1$. Also note that the integral in \cite{KM.II} is twice ours, since we are integrating over one of the 
two connected components. 

Arguing as before, we obtain the following. 

\begin{prop}
(i)  
If $\P(x;\gamma)\ne 0$ then $D_x\cap S$ is non-empty for any choice of $S$. \hfb
(ii) If $x$ is regular with respect to $\gamma$, then $\P(x;\gamma)$ is an integer and 
$$[\gamma] = \varepsilon(X)\,\P(x;\gamma)\cdot [S_{x,z_1}] \ \in \ \pi_{q-1}(D^+-D^+_x,z_0).$$
(iii) If $x$ is regular with respect to $\gamma$, then 
$$I^0(x;\gamma) = \P(x;\gamma) + J^0(x;\gamma).$$
\end{prop}

At this level of generality, we do not have an explicit expression for $J^0(x;\gamma)$ like that in Proposition~\ref{J0-formula}.
The following assumption allows us to continue the argument.\hfb
{\bf Assumption A:}
Suppose that, for $x$ regular with respect to $\gamma$, 
\beq\label{ansatz}
J^0(x;\gamma) = \mathcal E(x;\gamma) -\mathcal D(x;\gamma),
\eeq
where $\mathcal E(x;\gamma)$ extends to a continuous function of $x\in V$ and $\mathcal D(x;\gamma)$ is valued in $\frac{1}{M}\Z$ for some integer $M\ge 1$. 

Then, the difference 
$$\P(x;\gamma) -\mathcal D(x;\gamma) = I^0(x;\gamma) - \mathcal E(x;\gamma)$$ 
takes values in a discrete set for $x$ regular, but also extends to a continuous function on $V$, hence is a constant 
$$\cbold(\gamma) =\P(x;\gamma) -\mathcal D(x;\gamma) = -\mathcal D(\vbold;\gamma),$$ 
for any negative vector $\vbold$ in $V$. 
We obtain 
$$I^0(x;\gamma) = \mathcal E(x;\gamma) + \cbold(\gamma)$$
and 
$$\P(x;\gamma)  = \mathcal D(x;\gamma)-\mathcal D(\vbold;\gamma).$$

One more piece of information is needed. \hfb
{\bf Assumption B:} Suppose that $|\mathcal D(x;\gamma)|\le B$ for some constant $B$.  

Under these assumptions, we obtain convergence and completion.\hfb 
{\bf (1)} The series 
\beq\label{general-mock}
\vartheta_\mu(\tau;\gamma)=\sum_{x\in \mu+L} \P(x;\gamma) \, \qbold^{Q(x)}
\eeq 
is termwise absolutely convergent. For terms in the series with $x$ regular with respect to $\gamma$, the coefficients are generalized linking numbers. \hfb
{\bf (2)} The modular completion of $\vartheta_\mu(\tau;\gamma)$ is given by 
\begin{align}\label{general-completion}
\hat\vartheta_{\mu}(\tau;\gamma) &= \sum_{x\in \mu+L} (\mathcal E(v^{\frac12}x;\gamma) + \cbold(\gamma)\,)\, \qbold^{Q(x)}\\
\nass
{}&= \vartheta_\mu(\tau;\gamma) + \sum_{x\in \mu+L} (\mathcal E(v^{\frac12}x;\gamma) - \mathcal D(v^{\frac12}x;\gamma)\ )\,\qbold^{Q(x)}.\notag
\end{align}
The modularity of this completion follows from the identity
$$\hat\vartheta_{\mu}(\tau;\gamma)=\int_S \theta_\mu(\tau,\ph_{KM}).$$

{\bf Remark:} In all cases handled so far, an explicit formula for $J^0(x;\gamma)$ confirming Assumptions~A and B was obtained inductively. The induction depends on the 
fact that the `faces' of $\gamma$ lie in subsymmetric spaces of the form 
$$D'_y= \{ z\in D\mid y\in z\ \}$$
for negative vectors $y$ associated to the subspaces $y^\perp$ of $V$ of signature $(p,q-1)$. Examples for this were given in \cite{FK-II} where we discussed the cases when the collection $\CC$ gives to geodesic $q$-`cubes' respectively $q$-simplices in the symmetric space $D$. It would be very interesting to see if further examples can be obtained without this inductive structure.



\section{The dodecahedron}\label{section-dod}

In this section, we consider an example for which explicit formulas can again be obtained inductively, beginning with the result above on the $5$-gon.

Let $V$ be a quadratic over $\Q$ with $\sig(V) = (m-3,3)$.  For a negative vector $y$, the space $V_y= y^\perp$ has 
$\sig(V_y) = (m-3,2)$ and a collection of vectors in $V_y$ satisfying the $5$-gon conditions determine a $5$-gon in $D(V_y)$. 
The idea now is to construct a dodecahedron in $D(V)$ with such $5$-gons as its faces. 

Let 
$$\Cdod = \{\, C_0, \dots, C_{11}\,\}$$
be a collection of negative vectors in $V$. Let $V_j= C_j^\perp$ and $D'_j = D(V_j) \subset D(V)$.  The faces of the dodecahedron will be $5$-gons in 
the spaces $D'_j$ determined by the projections to $V_j$ of suitable ordered subsets of $\mathcal C$ that we now describe. 

We label the faces of a dodecahedron by elements $j$ in $\Z/12\Z$.  Choose an initial face $F_0$ (the `top') labeled $0$. Label the $5$ adjacent faces 
$F_1$, $F_2$, $F_3$, $F_4$, $F_5$, clockwise with respect to the outward normal. 
Define an involution of $\Z/12\Z$ by $a\mapsto \bar a = -(a+1)$.  Faces occur in antipodal pairs and we require that the labels of such a pair are $a$ and $\bar a$. 
The numbering of the faces is then determined. 
In particular, the `bottom' face, opposite to the `top' $F_0$, will be $F_{\bar 0}=F_{11}$. 
The cycle of faces adjacent to the top (resp. bottom) face will be numbered
$$\FF(0) = (1,2,3,4,5),$$
and 
$$\FF(\bar 0)= (\bar 5, \bar 4, \bar 3, \bar 2, \bar 1).$$
Note that the order is clockwise with respect to the outward normal in both cases. The cycles for the remaining faces are 
\begin{align*}
\FF(1)&=(0,5,\bar3, \bar4,2)&&\FF({\bar1})=(\bar0,\bar2,4,3,\bar5)\\
\FF(2)&=(0,1,\bar4,\bar5,3)&& \FF({\bar2})=(\bar0,\bar3,5,4,\bar1)\\
\FF(3)&=(0,2,\bar5,\bar1,4)&&\FF({\bar3})=(\bar0,\bar4,1,5,\bar2)\\
\FF(4)&=(0,3,\bar1,\bar2,5)&&\FF({\bar4})=(\bar0,\bar5,2,1,\bar3)\\
\FF(5)&=(0,4,\bar2,\bar3,1)&&\FF({\bar5})=(\bar0,\bar1,3,2,\bar4).
\end{align*}
Here the cycle $\FF({\bar j})$ is obtained from $\FF(j)$ by taking $\bar{\phantom a}$ and reversing the order. 
The face $F_j$ is adjacent to $F_i$ if and only if $j\in \FF(i)$, and, in that case, the cycle $\FF(j)$ is obtained from $\FF(i)= (a, j, b, u,v)$ by the recipe:
\beq\label{recipe}
\FF(i)=(a, j, b, u,v)\ \mapsto (b,i,a,\bar u, \bar v) =\FF(j).
\eeq
For example, 
$$\FF(4) = (\bar 2, 5, 0,3, \bar 1) \mapsto (0,4,\bar2,\bar3, 1) = \FF(5).$$
The recipe (\ref{recipe}) shows that,  if $F_i$ and $F_j$ are adjacent, the sets $\FF(i)$ and $\FF(j)$ share precisely $2$ elements, say $a$ and $b$,  so that $\FF(i)$ contains the 
subsequence $(a,j,b)$, and $\FF(j)$ contains the subsequence $(b,i,a)$. 

Each vertex is incident to a triple of faces, which we list clockwise with respect to the outward normal. 
The $5$ vertices of a face $F_i$ correspond to successive pairs of elements in $\FF(i)$, and, if  $(u,v)$ is such a pair, the corresponding 
vertex is labeled $[i,u,v]$.  If the faces $F_i$ and $F_j$ are adjacent, then they share vertices $[i,a,j]=[j,i,a]$ and $[j,b,i]=[i,j,b]$.

Recall that an ordered $5$-tuple of vectors $(v_1,v_2,v_3,v_4,v_5)$ in an inner product space $W$ of signature $(p,2)$ is said to satisfy {\bf the $5$-gon conditions} if 
\begin{align}
(v_j,v_j)&<0\notag\\
\nass
(v_j,v_j)(v_{j+1},v_{j+1}) - (v_j,v_{j+1})^2&>0\label{5-gon}\\
\nass
(v_j,v_j)(v_{j-1},v_{j+1}) - (v_j,v_{j-1})(v_j,v_{j+1}) &<0.\notag
\end{align}
We write  
$$P_j: V \lra V_j,  \qquad P_j v = v - \frac{(v,C_j)}{(C_j,C_j)}\,C_j$$
for the orthogonal projection to $V_j= C_j^\perp$. 

We say that a collection $\Cdod = \{ \, C_j\, \}_{j\in \Z/12\Z}$ satisfies the 
{\bf dodecahedron conditions} if, for each $i\in \Z/12\Z$, the ordered $5$-tuple of vectors  
\beq\label{proj-5-tuple}
\mathcal R(i)=(P_i C_j)_{j\in \FF(i)} = (\RR(i)_1, \RR(i)_2,\RR(i)_3,\RR(i)_4,\RR(i)_5)
\eeq
satisfies the $5$-gon conditions (\ref{5-gon}).  Here the ordering in (\ref{proj-5-tuple}) is given by the ordering in $\FF(i)$ so that, for example, for $F_{\bar2}$ the 
$5$-tuple in $V_{\bar2}$ is
$$\RR(\bar2)=(P_{\bar2} C_{\bar0}, P_{\bar2}C_{\bar3}, P_{\bar2}C_5, P_{\bar2}C_4, P_{\bar2}C_{\bar1}).$$
If $\Cdod$ is such a collection, then each $5$-tuple (\ref{proj-5-tuple}) defines a $5$-gon $\gamma(i)$ in the 
space $D(V_i)$ where we recall that $\sig(V_i) = (m-3,2)$.  

More explicitly, suppose that  $j\in \FF(i)$ and write $\FF(i)=(a,j,b,c,d)$.
Note that the orthogonal complement in $V_i$ of the vector $P_i C_j$ is precisely\footnote{If $v\in V_i$, then $(v,P_iC_j) = (v,C_j)$.} $V_i\cap V_j$. 
Let 
$$Q^i_j: V_i \lra V_i\cap V_j \qquad Q^i_j v = v- \frac{(v,P_iC_j)}{(P_iC_j,P_i C_j)}\, P_i C_j.$$
Define the path $\gamma(i)_j$  by 
$\gamma(i)_j:[0,1] \lra D(V)$, 
$$s \mapsto [C_i, P_iC_j, (s-1)Q^i_j(P_iC_a)+ s\, Q^i_j(P_i C_b)] = [C_i, C_j, (s-1)C_a+s C_b].$$
Here note that, by the $5$-gon condition, $[P_iC_j, (s-1)Q^i_j(P_iC_a)+ s\, Q^i_j(P_i C_b)]$ is an oriented negative $2$-plane in $V_i$ and hence
$\gamma(i)_j(s)$ is an oriented negative $3$-plane in $V$.   The endpoints of this path are
$$\gamma(i)_j(0) = [C_j,C_i,C_a]$$
and 
$$\gamma(i)_j(1) =  [C_i,C_j,C_b],$$
oriented negative $3$-planes in the same component of $D(V)$. 
Then the $5$-gon $\gamma(i)$ is the sum of paths  $\gamma(i)_j$ given by this recipe as $j$ runs over $\FF(i)$. 

Now suppose that $F(i)$ and $F(j)$ are adjacent faces with $\FF(i)= (a,j,b,c,d)$ and $\FF(j) = (b,i,a,\bar c, \bar d)$.
The $5$-gon $\gamma(j)$ in $D(V_j)$ includes a segment 
$$\gamma(j)_i(s) = [C_j, P_iC_i, (s-1)Q^j_i(P_jC_b)+ s\, Q^j_i(P_j C_a)] = [C_j, C_i, (s-1)C_b +s\, C_a].$$
Writing 
$$\gamma(j)_i(s) = [C_i,C_j, -s C_a + (1-s) C_b]$$
we see that 
$$\gamma(j)_i(s) = \gamma(i)_j(1-s),$$
i.e., these two `edges' coincide, with opposite parametrization. 

Thus we have proved the following result. 
\begin{prop}  Suppose that a collection $\Cdod$ satisfies the dodecahedral conditions. \hfb 
(i)  The $5$-gons in the $D(V_i)$'s associated to the collections 
$(P_i C_j)_{j\in \mathcal F(i)}$
all lie in the same component of $D(V)$. Denote this component by $D(V)^+$. \hfb
(ii) For each $i$, let $S(i)$ be an oriented $2$-cell in $D(V_i)^+$ with $\d S(i) = \gamma(i)$. 
Then the $2$-chain 
$$S(\Cdod) = \sum_{i=0}^{11} S(i)$$
in $D(V)^+$ has $\d S(\Cdod) =0$. 
\end{prop}

\begin{cor} There is an oriented $3$-chain $\Ddod$ in $D(V)^+$ with $\d \mathcal D^{\text{dod}} = S(\mathcal C^{\text{dod}}).$
\end{cor}

We may then define a function of $x\in V$ by 
$$I^0(x;\mathcal C^{\text{dod}}) = \int_{\mathcal D^{\text{dod}}}\ph^0_{KM}(x).$$
Once the $2$-cells $S(i)$ have been chosen, $I^0(x;\mathcal C^{\text{dod}})$ does not depend on the choice of the $3$-chain $\mathcal D^{\text{dod}}$.
In fact, we will see that there is also no dependence on the $2$-cells $S(i)$ in $D(V_i)$. 

The regular set is 
$$\text{Reg}(\mathcal C^{\text{dod}})=\{\ x\in V\mid (x,C_j)\ne0, \ \forall j\, \}.$$
If $x$ is regular, the $5$-gons $\gamma(i)$ and the $2$-cells $S(i)$ lie in $D(V)- D_x$, since they lie in the subsymmetric spaces 
$D(V_i) \hookrightarrow D(V)$, where the embedding 
is given by $\zeta \mapsto [C_i,\zeta]$.

Recall that, on the set $D(V)-D_x$, 
 $$\ph^0_{KM}(x) = d\Psi_{KM}^0(x),$$
for an explicit $2$-form $\Psi_{KM}^0(x)$ on $D$.
For $x$ is regular with respect to $\Cdod$, let 
$$J^0(x;\Cdod):= \int_{S(\Cdod)} \Psi^0_{KM}(x).$$

As before, an inductive calculation yields the following. 

\begin{prop}  For $x$ regular with respect to $\Cdod$, 
\begin{align*}
J^0(x;\Cdod) & = \frac18\sum_{\nubold} E_3(\nubold,  x\sqrt{2}) - \sgn(x;\nubold)\\
\nass
{}&\qquad+ \frac18\sum_{i\in \Z/12\Z} \wnat(\RR(i)) \bigg(\ E_1(C_i;x\sqrt{2}) - \sgn((x,C_i))\ \bigg)
\end{align*}
where $\nubold$ runs over the vertices of the dodecahedron and where, with $\RR(i)$ given by (\ref{proj-5-tuple}), 
$$\wnat(\RR(i)) = - \sum_{\ell =1}^5 \sgn((\vbold_i,\RR(i)_\ell))\,\sgn((\vbold_i,\RR(i)_{\ell+1}))$$
for any negative vector $\vbold \in V_i$. 
\end{prop}
\begin{proof}
Using the formulas from \cite{FK-II}, we have  
\begin{align*}
\int_{S(\Cdod)} \Psi^0_{KM}(x) & = \sum_i \int_{S(i)} \Psi^0_{KM}(x)\\
\nass
{}&= -\sum_i\sqrt{2}\,(x,\und{C_i})\int_1^\infty e^{-2\pi t^2(x,\und{C_i})^2}\int_{S(i)} \ph_{KM}^{V_j,0}(t x_{\perp C_i}) \,dt\\
\nass
{}&= -\sum_i \sqrt{2}\,(x,\und{C_i})\int_1^\infty e^{-2\pi t^2(x,\und{C_i})^2} I^0(t x_{\perp C_ i}; \RR(j)) \,dt\\
\nass
{}&= -\frac14\sum_i (\sqrt{2}x,\und{C}_i) \int_1^\infty e^{-2\pi t^2(x,\und{C_i})^2} \\
\nass
&\qquad\qquad \times \bigg[\ \wnat(\RR(i)) + \sum_{\ell =1}^5 E_2(\RR(i)_\ell, \RR(i)_{\ell+1},tx_{\perp C_i}\sqrt{2})  \bigg]\,dt.
\end{align*}
Here in the first step we use the relation of Corollary~6.3,
$$\kappa_j^*\Psi_{KM}^0(x) = - \sqrt{2}\,(x, \und{C}_j)\,\int_1^\infty e^{-2\pi t^2 (x, \und{C}_j)^2} \ph_{KM}^{V_j,0}(t x_{\perp C_j})\,dt,$$

To apply the induction identity, we regroup the $60$ terms in the last expression in $3$'s according to the vertices.
For example, the vertex $[0,1,2]$ is shared by the faces $F_0$, $F_1$ and $F_2$.  The corresponding $5$-gon data are
\begin{align*}
\RR(0)&= (\ P_0C_1, P_0 C_2, P_0C_3, P_0C_4, P_0C_5\ )\\
\nass
\RR(1) &= (\ P_1 C_0,P_1C_5, P_1 C_{\bar3}, P_1C_{\bar4}, P_1C_2\ )\\
\nass
\RR(2)&=(\ P_2 C_0,P_2 C_1, P_2 C_{\bar 4}, P_2 P_{\bar 5}, P_2 C_3\ ).
\end{align*}
We collect a term from each of the groups $i=0$, $1$ and $2$ amounting to the contribution
\begin{align*}
&-\frac14\,(\sqrt{2}\,x, \und{C}_0)\, \int_1^\infty e^{-2\pi t^2(x,\und{C}_0))^2}\,E_2(P_0C_1, P_0C_2,tx_{\perp C_0}\sqrt{2})\,dt\\
\nass
&-\frac14\,(\sqrt{2}\,x, \und{C}_1)\, \int_1^\infty e^{-2\pi t^2(x,\und{C}_1))^2}\,E_2( P_1C_0, P_1C_2,tx_{\perp C_1}\sqrt{2})\,dt\\
\nass
&-\frac14\,(\sqrt{2}\,x, \und{C}_2)\, \int_1^\infty e^{-2\pi t^2(x,\und{C}_2))^2}\,E_2(P_2C_0, P_2C_1,tx_{\perp C_0}\sqrt{2})\,dt,
\end{align*}
where, in the second line, we have used the symmetry $E_2(c_1,c_2;x) = E_2(c_2,c_1;x)$ of the generalized error function.  
By the recursion formula of Proposition~7.3 of \cite{FK-II}, this contribution is 
$$\frac18\,\bigg(\ E_3(C_0,C_1,C_2; x\sqrt{2}) - \sgn((x,C_0))\,\sgn((x,C_1))\,\sgn((x,C_2))\ \bigg).$$

Thus the total contribution of such terms is 
$$\frac18\sum_{\nubold} E_3(\nubold,  x\sqrt{2}) - \sgn(x;\nubold),$$
where the sum runs over the triples of vectors associated to the vertices of $\Ddod$. 

The remaining contribution is 
\begin{align*}
-\frac14&\sum_i \wnat(\RR(i))\,\sqrt{2}\,(x,\und{C}_i) \int_1^\infty e^{-2\pi t^2(x,\und{C_i})^2} \,dt\\
\nass
{}&=\frac18\sum_i \wnat(\RR(i)) \bigg(\ E_1(C_i;x\sqrt{2}) - \sgn((x,C_i))\ \bigg).
 \end{align*}
where $\wnat(\RR(i))$ is as in the proposition. 
\end{proof}

Now the topological argument of the previous section can be applied.  
%
%
The resulting modularity statement for the dodecahedron is summarized 
as follows. 

\begin{theo}  Suppose that $\sig(V)= (m-3,3)$ and that $\Cdod = \{ \,C_0, \dots, C_{11}\,\}$ is a collection of negative vectors in $V$ 
satisfying the dodecahedron conditions. 
Let
$$\P(x;\Cdod) = \mathcal D(x;\Cdod)-\mathcal D(\vbold;\Cdod),$$
where $\vbold$ is any negative vector in $V$ and
$$\mathcal D(x;\Cdod)  =- \frac18\sum_{\nubold}  \sgn(x;\nubold)\\
- \frac18\sum_{i\in \Z/12\Z} \wnat(\RR(i)) \,\sgn((x,C_i)).$$
Here $\nubold$ runs over the triples of vectors associated to the vertices of $\Ddod$, and for the
`face cycle'  $\FF(i) = \{\, j_1,j_2,j_3,j_4,j_5\,\}$, and collection 
$$\mathcal R(i) = \{ \, P_i C_{j_1}, P_i C_{j_2},P_i C_{j_3},P_i C_{j_4},P_i C_{j_5}\,\}$$
of negative vectors in $V_i = C_i^\perp$ associated to the $i$th face,  
$$\wnat(\RR(i)) = -\sum_{\ell=1}^5 \sgn((\vbold_i, C_{j_\ell}))\,\sgn((\vbold_i, C_{j_{\ell+1}})),$$
where $\vbold_i$ is any negative vector in $V_i$. 
Let 
$$\mathcal E(x;\Cdod) = \frac18\sum_{\nubold} E_3(\nubold,  x\sqrt{2})
+ \frac18\sum_{i\in \Z/12\Z} \wnat(\RR(i)) \,E_1(C_i;x\sqrt{2}).$$
%

(i) The integral $I^0(x;\Cdod)$ is given by 
$$I^0(x;\Cdod) = \mathcal E(x;\Cdod) - \mathcal D(\vbold;\Cdod)$$
for any negative vector $\vbold$ in $V$. In particular, the integral depends only on the collection $\Cdod$ and not on the choice of the faces surfaces $S(i)$  in $D'_{C_i}$. 
\hfb
(ii) 
For a lattice $L\subset L^\vee\subset V$, the series 
\beq\label{mock-dod}
\vartheta_\mu(\tau;\Cdod)=\sum_{x\in \mu+L} \P(x;\Cdod) \, \qbold^{Q(x)}
\eeq 
is termwise absolutely convergent and defines a mock modular form of weight $\frac12 m$ with completion
\begin{align}
\hat\vartheta_{\mu}(\tau;\Cdod) &= \sum_{x\in \mu+L} (\mathcal E(v^{\frac12}x;\Cdod) - \mathcal D(\vbold;\Cdod)\,)\, \qbold^{Q(x)}\\
\nass
{}&= \vartheta_\mu(\tau;\Cdod) + \sum_{x\in \mu+L} (\mathcal E(v^{\frac12}x;\Cdod) - \mathcal D(x;\Cdod)\ )\,\qbold^{Q(x)}.\notag
\end{align}
The modularity of this completion follows from the identity
$$\hat\vartheta_{\mu}(\tau;\Cdod)=\int_{\Ddod} \theta_\mu(\tau,\ph_{KM}).$$

\end{theo}

\begin{rem} (i) It is probably difficult to prove the absolute convergence of (\ref{mock-dod}) directly from the set of inequalities imposed by the 
dodecahedral conditions, a collection $5$-gon conditions on each of the 12 sets of projections $\mathcal R(i)$. \hfb
(ii) Of course, we do not know whether or not such dodecahedrons could arise in the context of \cite{ABMP-II}.  For us the main point is to provide an example in which the method of section~7 can be made to yield an explicit formula. The cube and tetrahedron were handled in \cite{FK-II}. The octahedron should be similar but simpler while the icosahedron will require bit more work. 
\end{rem}

A supply of collections satisfying the dodecahedron conditions is provided by a  `seed' construction. 
Fix an oriented negative $3$-plane $z_0$ with properly oriented orthonormal basis $u_0$, $u_1$, $u_2$ so that $z_0\simeq \R^3$ with the negative of the 
standard Euclidean inner product. For a regular dodecahedron $\mathbb D$ in $\R^3$, centered at the origin, let $\Cdod_0=\{\, c_0, \dots, c_{11}\,\}$ be the images in $z_0$
of the outward unit normal vectors to the faces, numbered as described above. 
In particular, $c_{\bar r} = - c_r$.   It is then easily checked, using $\SO(3)$ equivariance  and the transitivity of the symmetry group of the 
regular dodecahedron on faces, edges and vertices,  that 
the collection $\Cdod_0$ satisfies the dodecahedron conditions. Of course, the associated `dodecahedron' in $D=D(V)$ degenerates to the point $z_0$.
For a positive vector $v_0$ in $z_0^\perp$ and $t= (t_0,\dots,t_{11})\in \R^{12}$, let
$$\Cdod_t = \Cdod_0 + t v_0.$$
Since the dodecahedron conditions define an open subset of $V^{12}$, these conditions are still satisfied by $\Cdod_t$ for $t$ small. 
It remains to show that we can take $t$ so that the $20$ vertices determined by $\Cdod_t$ are distinct. 
But the coincidence of any two of these $3$-planes amounts to non-trivial linear conditions on the components of $t$, so the vertices will be distinct 
for an open set of $t$'s.

We leave aside the following problem.  Let $V^{12}_{\text{\rm dod}}$ be the set of collection $\Cdod$ satisfying the dodecahedral conditions, and let 
$(V^{12}_{\text{\rm dod}})^0$ be the subset of collection for which the associated $20$ vertices are distinct.  The problem is to write down additional conditions 
defining $(V^{12}_{\text{\rm dod}})^0$ and to determine its connected components. Is there a nice characterization of the geometry of the dodecahedra associated to the different\footnote{At the moment we do not know whether or not $(V^{12}_{\text{\rm dod}})^0$ is connected.}
components?

\end{document}

%% file: KRY.macros.tex



\newcommand{\A}{{\mathbb A}}
\newcommand{\C}{{\mathbb C}}
\newcommand{\F}{{\mathbb F}}
\newcommand{\G}{{\mathbb G}}
\newcommand{\R}{{\mathbb R}}
\newcommand{\Q}{{\mathbb Q}}
\newcommand{\X}{{\mathbb X}}
\newcommand{\Z}{{\mathbb Z}}
\newcommand{\HZ}{\widehat{\Z}}


\newcommand{\rom}[1]{\text{\rm #1}}
\renewcommand{\roman}{\rm}

\newcommand{\Aut}{\text{\rm Aut}}
\newcommand{\CH}{\widehat{\text{\rm CH}}}
\newcommand{\cha}{{\text{\rm char}}}
\newcommand{\CHe}{\text{\rm CHeeg}}
\newcommand{\degh}{\widehat{\text{\rm deg}}}
\newcommand{\degH}{\widehat{\text{\rm deg}}}    
\newcommand{\diag}{{\text{\rm diag}}}
\newcommand{\Diff}{\text{\rm Diff}}
\newcommand{\disc}{\text{\rm discr}}
\renewcommand{\div}{\text{\rm div}}
\newcommand{\divh}{\widehat{\text{\rm div}}}
\newcommand{\DS}{\text{\rm DS}}
\newcommand{\Ei}{\text{\rm Ei}}
\newcommand{\End}{\text{\rm End}}
\newcommand{\ev}{{\text{\rm ev}}}
\newcommand{\Gal}{\text{\rm Gal}}
\newcommand{\GL}{\text{\rm GL}}
\newcommand{\GSpin}{\text{\rm GSpin}}
\newcommand{\Hom}{\text{\rm Hom}}
\newcommand{\hor}{{\text{\rm horiz}}}
\newcommand{\id}{\text{\rm id}}
\newcommand{\im}{\text{\rm im}}
\renewcommand{\Im}{\text{\rm Im}}
\newcommand{\inv}{{\text{\rm inv}}}
\newcommand{\Jac}{\text{\rm Jac}}
\newcommand{\Leray}{{\mathrm L}}
\newcommand{\Lie}{\text{\rm Lie}}
\newcommand{\Mp}{\text{\rm Mp}}
\newcommand{\mult}{\text{\rm mult}}
\newcommand{\MW}{\text{\rm MW}}
\newcommand{\MWt}{\widetilde{\MW}}
\newcommand{\new}{\text{\rm new}}
\newcommand{\Nm}{\text{\rm Nm}}
\newcommand{\ord}{\text{\rm ord}}
\newcommand{\PGL}{\text{\rm PGL}}
\newcommand{\Pic}{\text{\rm Pic}}
\newcommand{\Pich}{\widehat{\text{\rm Pic}}}
\newcommand{\pr}{\text{\rm pr}}
\newcommand{\ra}{\text{\rm ra}}
\newcommand{\Rao}{\mathrm R}
\renewcommand{\Re}{\text{\rm Re}}
\newcommand{\sgn}{\text{\rm sgn}}
\newcommand{\sig}{\text{\rm sig}}
\newcommand{\SL}{\text{\rm SL}}
\newcommand{\SO}{\text{\rm SO}}
\newcommand{\Sp}{\text{\rm Sp}}
\newcommand{\Spec}{\text{\rm Spec}\, }
\newcommand{\Spf}{\text{\rm Spf}}
\newcommand{\supp}{\text{\rm supp}}
\newcommand{\Sym}{{\text{\rm Sym}}}
\newcommand{\tr}{\text{\rm tr}}
\newcommand{\type}{\text{\rm type}}
\newcommand{\Ver}{\text{\rm Vert}}
\newcommand{\vol}{\text{\rm vol}}
\newcommand{\Wald}{\text{\rm Wald}}


\newcommand{\Cal}{\mathcal}     

\newcommand{\AHH}{\hat{\Cal A}}   
\newcommand{\CHH}{\hat{\Cal C}}
\newcommand{\MM}{\Cal D}          
\newcommand{\MMb}{\MM^\bullet}
\newcommand{\ssplit}{\text{\bf split}}
\newcommand{\whcc}{\widehat{\Cal C}}
\newcommand{\CO}{\mathcal O}
\newcommand{\COH}{\widehat{\CO}}
\newcommand{\M}{\Cal M}
\newcommand{\OB}{\Cal O_B}
\newcommand{\XX}{\mathcal X}
\newcommand{\bXX}{\bar\XX}
\newcommand{\wc}{\hat{\Cal C}}
\newcommand{\wch}{\wc^{\text{\rm hor}}}
\newcommand{\ZZ}{\Cal Z}
\newcommand{\ZH}{\widehat{\Cal Z}}   
\newcommand{\Zh}{\widehat{\Cal Z}}
\newcommand{\ZZh}{\ZZ^{\text{\rm hor}}}
\newcommand{\ZZv}{\ZZ^{\text{\rm ver}}}
\newcommand{\ZZhh}{\Zh^{\text{\rm hor}}}
\newcommand{\ZZhv}{\Zh^{\text{\rm ver}}}


\newcommand{\nass}{\noalign{\smallskip}}
\newcommand{\snass}{\noalign{\vskip 2pt}}
\newcommand{\tent}[1]{ \vphantom{\vbox to #1pt{}} }   


\newcommand{\scr}{\scriptstyle}
\newcommand{\disp}{\displaystyle}

\font\cute=cmitt10 at 12pt
\font\smallcute=cmitt10 at 9pt
\newcommand{\kay}{{\text{\cute k}}}
\newcommand{\smallkay}{{\text{\smallcute k}}}

\renewcommand{\a}{\alpha}
\renewcommand{\b}{\beta}
\newcommand{\e}{\epsilon}
\renewcommand{\l}{\lambda}
\renewcommand{\L}{\Lambda}
\renewcommand{\o}{\omega}
\renewcommand{\O}{\Omega}
\renewcommand{\P}{\Phi}
\newcommand{\ph}{\varphi}
\newcommand{\phih}{\widehat{\phi}}
\newcommand{\wphi}{\widehat{\phi}}
\newcommand{\phit}{\widetilde{\phi}}
\newcommand{\s}{\sigma}
\newcommand{\vth}{\vartheta}


%

\newcommand{\Pt}{P}
\newcommand{\Ph}{\P}
\newcommand{\Pht}{\tilde \P}   
\newcommand{\Kt}{K}           
\newcommand{\Mt}{M}

\newcommand{\pht}{\widetilde{\phi}}
\newcommand{\It}{I}
\newcommand{\Jt}{\widetilde{J}}
\newcommand{\lt}{\widetilde{\l}}
\newcommand{\vp}{\varpi}

\newcommand{\bom}{{\boldsymbol{\o}}}
\newcommand{\hbom}{\widehat{\bom}}
\newcommand{\ff}{{\bold f}}
\newcommand{\fsp}{\boldsymbol{f}_{\!\rm sp}}
\newcommand{\fev}{\boldsymbol{f}_{\!\rm ev}}
\newcommand{\fb}{\boldsymbol{f}}
\newcommand{\J}{\und{J}'}
\newcommand{\JJ}{\bold J'}
\newcommand{\V}{\bold V}
\newcommand{\xx}{\bold x}

\newcommand{\g}{{\mathfrak g}}
\renewcommand{\H}{\mathfrak H}


\newcommand{\back}{\backslash}
\newcommand{\CT}[1]{\operatornamewithlimits{CT}_{#1}}
\renewcommand{\d}{\partial}
\newcommand{\db}{\bar\partial}
\newcommand{\dbar}{\bar{\partial}}
\newcommand{\gs}[2]{\langle \,#1,#2\,\rangle}
\newcommand{\Gt}{G}
\newcommand{\hfal}{h_{\text{\rm Fal}}}
\newcommand{\II}{\int^\bullet}
\newcommand{\isoarrow}{\ {\overset{\sim}{\longrightarrow}}\ }
\newcommand{\lisoarrow}{\ {\overset{\sim}{\longleftarrow}}\ }
\newcommand{\limdir}[1]{\underset{\underset{#1}{\rightarrow}}{\lim}}
\newcommand{\lan}{\operatorname{\langle}\hskip .5pt}
\newcommand{\ran}{\,\operatorname{\rangle}}
\newcommand{\lra}{\longrightarrow}
\newcommand{\doublelra}{\ {\overset{\scr\lra}{\scr\lra}}\ }
\newcommand{\nat}{\natural}
\newcommand{\notmid}{\mkern-5mu\not\mkern5mu\mid}
\newcommand{\Optoc}{\text{\rm Opt}(O_{c^2d},O_B)}
\newcommand{\psim}{\psi^{-}}
\newcommand{\qeq}{\ \overset{??}{=}\ }
\newcommand{\sh}{\sharp}
\newcommand{\thCH}{\theta^{\text{\rm ar}}}
\newcommand{\wht}{\widehat{\theta}}     
\newcommand{\triv}{1\!\!1}
\renewcommand{\tt}{\otimes}
\newcommand{\und}[1]{\underline{#1}}
\newcommand{\z}{z}  

\newcommand{\thMW}{\theta^{\text{\rm ar}}}
\newcommand{\tph}{\widetilde{\widehat\phi_1}}
\newcommand{\Pet}{\text{\rm Pet}}





\newcommand{\thing}{ \raisebox{-6.4pt}{$\tilde{\tilde{}}$}  }   
\newcommand{\smallthing}{ \raisebox{-4.4pt}{$\scr\tilde{\tilde{}}$}  }
\newcommand{\ttilde}[1]{\overset{\smash{\thing}}{#1}}
\newcommand{\smallttilde}[1]{\overset{\smash{\smallthing}}{#1}}
\newcommand{\downhookarrow}{\hbox{$\downarrow\hskip -6.1pt\raisebox{6pt}{$\cap$}$}}


\providecommand{\bysame}{\makebox[3em]{\hrulefill}\thinspace}   
\newcommand{\hfb}{\hfill\break}
\newcommand{\margincom}[1]{\marginpar{\bf\raggedright #1}}
\newcommand{\Sec}{\S}


\numberwithin{equation}{section}
\setcounter{section}{0}
\setcounter{MaxMatrixCols}{15}


\newtheorem{theo}{Theorem}[section]
\newtheorem{lem}[theo]{Lemma}
\newtheorem{prop}[theo]{Proposition}
\newtheorem{cor}[theo]{Corollary}
\newtheorem{conj}[theo]{Conjecture}
\newtheorem{rem}[theo]{Remark}      
\newtheorem{defn}[theo]{Definition}